\documentclass[11pt]{article}
\usepackage[T1]{fontenc}
\usepackage{amsfonts}
\usepackage{amsmath}
\usepackage{amssymb}
\usepackage{amsthm}
\usepackage{bbm}
\usepackage{bm}
\usepackage{mathrsfs}
\usepackage{verbatim}
\usepackage{setspace}
\usepackage{color}
\usepackage{pdfsync}
\usepackage{enumitem}
\usepackage{graphicx}
\usepackage{subfigure}
\usepackage{tikz}
\usetikzlibrary{patterns}

\theoremstyle{plain}
\newtheorem{theorem}{Theorem}[section]
\newtheorem{proposition}[theorem]{Proposition}
\newtheorem{lemma}[theorem]{Lemma}
\newtheorem{corollary}[theorem]{Corollary}

\theoremstyle{definition}
\newtheorem{definition}[theorem]{Definition}
\newtheorem{remark}[theorem]{Remark}

\newtheorem{example}[theorem]{Example}

\newtheorem{assumption}[theorem]{Assumption}

\theoremstyle{remark}

\renewenvironment{thebibliography}[1]{%
\begin{oldthebibliography}{#1}%
\setlength{\baselineskip}{1em}
\linespread{.2}
\small
\setlength{\parskip}{0.25ex}%
\setlength{\itemsep}{.20em}%
}%
{%
\end{oldthebibliography}%
}
\newcommand{\eps}{\varepsilon}

\newcommand{\F}{\mathbb{F}}

\newcommand{\N}{\mathbb{N}}
\newcommand{\Q}{\mathbb{Q}}
\newcommand{\R}{\mathbb{R}}

\newcommand{\W}{\mathbb{W}}

\newcommand{\cB}{\mathcal{B}}
\newcommand{\cC}{\mathcal{C}}
\newcommand{\cD}{\mathcal{D}}
\newcommand{\cE}{\mathcal{E}}
\newcommand{\cF}{\mathcal{F}}

\newcommand{\cL}{\mathcal{L}}

\newcommand{\cR}{\mathcal{R}}

\newcommand{\cT}{\mathcal{T}}
\newcommand{\cU}{\mathcal{U}}
\newcommand{\cV}{\mathcal{V}}

\newcommand{\cY}{\mathcal{Y}}

\newcommand{\bC}{\mathbf{C}}

\newcommand{\bI}{\mathbf{I}}

\newcommand{\bS}{\mathbf{S}}
\newcommand{\bt}{\boldsymbol{t}}

\DeclareMathOperator{\USA}{USA}

\DeclareMathOperator{\supp}{supp}

\DeclareMathOperator{\Law}{Law}
\newcommand{\1}{\mathbf{1}}

\numberwithin{equation}{section}

\usepackage[pdfborder={0 0 0}]{hyperref}
\hypersetup{
  urlcolor = black,
  pdfauthor = {Mathias Beiglbock, Marcel Nutz, Florian Stebegg},
  pdfkeywords = {Skorokhod Embedding Problem; Randomized Stopping Time; Kantorovich Duality},
  pdftitle = {Fine Properties of the Optimal Skorokhod Embedding Problem},
  pdfsubject = {Fine Properties of the Optimal Skorokhod Embedding Problem},
  pdfpagemode = UseNone
}

\begin{document}

\title{\vspace{-1em}
Fine Properties of the Optimal Skorokhod Embedding Problem\footnote{The authors thank two anonymous referees for their detailed comments.}
\date{\today}
\author{
  Mathias Beiglb\"ock%
  \thanks{University of Vienna, Department of Mathematics, mathias.beiglboeck@univie.ac.at. Research supported by FWF Grant Y-782.}
  \and
  Marcel Nutz%
  \thanks{
  Columbia University, Departments of Statistics and Mathematics, mnutz@columbia.edu. Research supported by an Alfred P.\ Sloan Fellowship and NSF Grants DMS-1512900 and DMS-1812661. MN is grateful to Alex Cox and Johannes Ruf for helpful discussions regarding Sections~\ref{se:cave} and~\ref{se:localMartEx}, respectively.
  }
  \and
  Florian Stebegg%
  \thanks{
  Columbia University, Department of Statistics, florian.stebegg@columbia.edu.
  }  
 }
}
\maketitle \vspace{-1.2em}

\begin{abstract}
We study the problem of stopping a Brownian motion at a given distribution $\nu$ while optimizing a reward function that depends on the (possibly randomized) stopping time and the Brownian motion. Our first result establishes that the set $\cT(\nu)$ of stopping times embedding~$\nu$ is weakly dense in the set $\cR(\nu)$ of randomized embeddings. In particular, the optimal Skorokhod embedding problem over $\cT(\nu)$ has the same value as the relaxed one over $\cR(\nu)$ when the reward function is semicontinuous, which parallels a fundamental result about Monge maps and Kantorovich couplings in optimal transport. A second part studies the dual optimization in the sense of linear programming. While existence of a dual solution failed in previous formulations, we introduce a relaxation of the dual problem that exploits a novel compactness property and yields existence of solutions as well as absence of a duality gap, even for irregular reward functions. This leads to a monotonicity principle which complements the key theorem of Beiglb\"ock, Cox and Huesmann [Optimal transport and Skorokhod embedding, \emph{Invent.\ Math.}, 208:327--400, 2017]. We show that these results can be applied to characterize the geometry of optimal embeddings through a variational condition.
\end{abstract}

\vspace{1em}

{\small

}
\vspace{2em}

\section{Introduction}\label{se:intro}

Given a centered and suitably integrable probability distribution $\nu$ and a Brownian motion $B$, the Skorokhod embedding problem~\cite{Skorokhod.65} consists in finding a stopping time $\tau$ which embeds $\nu$; that is, $B_{\tau}$ has distribution~$\nu$. A number of solutions exist and we denote the set of all such $\tau$ by $\cT(\nu)$. Examples include the classical Root~\cite{Root.69} and Rost~\cite{Rost.76} embeddings; see \cite{Obloj.04} for a survey of various solutions. The \emph{optimal} Skorokhod embedding problem is to maximize (or minimize) the expectation $E[G_{\tau}]$ over $\tau\in \cT(\nu)$, where $G_{t}=G((B_{s})_{s\leq t},t)$ is an adapted functional. For instance, the Root embedding minimizes $E[\tau^{2}]$ and the Rost embedding maximizes it; cf.~\cite{Rost.76}. Some early works related to optimal Skorokhod embeddings are \cite{BrownHobsonRogers.01, Hobson.98, HobsonPedersen.02, MadanYor.02}. More recently, connections to numerous questions in probability, analysis and finance as well as extensions such as the multi-marginal case \cite{BeiglbockCoxHuesmann.17, CoxOblojTouzi.15, GuoTanTouzi.15a} have emerged  and led to substantial activity; we refer to~\cite{Hobson.11} for a survey with many more references.

A perspective pioneered by \cite{BeiglbockCoxHuesmann.14} is to see Skorokhod embeddings along the lines of optimal transport theory: optimal stopping times are analogues of Monge solutions to an optimal transport problem between the Wiener measure and~$\nu$. A more general formulation of the embedding problem allows for a randomized stopping time; this can be interpreted as using an enlarged filtration or allowing for an external randomization (see Definition~\ref{de:RST}). The corresponding set is denoted $\cR(\nu)$ and gives rise to a relaxed formulation of the optimal Skorokhod embedding problem. Continuing the analogy, randomized stopping times correspond to transports in the sense of Kantorovich. We refer to \cite{AmbrosioGigli.13,RachevRuschendorf.98a,RachevRuschendorf.98b,Villani.03,Villani.09} for background on optimal transport.

In the existing literature on the optimal Skorokhod embedding problem, a number of optimal embeddings have been found for specific reward functionals $G$; e.g., \cite{CoxObloj.11, Hobson.98, HobsonKlimmek.12, HobsonKlimmek.15, HobsonNeuberger.12}. In these examples, optimal embeddings are often unique and belong to the class $\cT(\nu)$ of stopping times even if one allows for randomized stopping times a priori. On the other hand, results concerning the general structure of the optimal Skorokhod embedding problem such as~\cite{BeiglbockCoxHuesmann.14, GuoTanTouzi.15,GuoTanTouzi.15a} use the formulation with randomized stopping times. Thus, an obvious---but not previously addressed---question is how to bridge this gap: when can the supremum value over randomized stopping times be achieved with stopping times? More generally, can randomized stopping times be approximated by stopping times in a suitable sense?

The analogy to classical optimal transport theory is apparent. While the Kantorovich relaxation is crucial to develop the theory, most examples of specific interest lead to transport maps in the sense of Monge. For instance, Brenier's theorem states that the optimal transport for the quadratic cost (or reward, after changing the sign) is given by the gradient of a convex function if the first marginal measure is absolutely continuous (or, more generally, regular \cite{McCann.95}). When the first marginal is atomless, it was shown in~\cite{Ambrosio.03} that Monge transports form a dense subset of Kantorovich couplings and that the values of the Monge and Kantorovich transport problems agree for bounded continuous reward functions. This result was extended to unbounded continuous functions in~\cite{Pratelli.07}; see also~\cite{Lacker.18a} for a survey of related density properties.

In the first part of this paper, we provide comparable results for the optimal Skorokhod embedding problem. In Theorem~\ref{th:MongeValue} we show that the optimal embedding problems over $\cT(\nu)$ and $\cR(\nu)$ have the same value whenever the reward functional $G$ is lower semicontinuous in time. This assertion can fail when $G$ is not lower semicontinuous (Example~\ref{ex:counterexApprox}), and that failure highlights a contrast with classical approximation results for randomized stopping times that hold without regularity conditions~\cite{BaxterChacon.77, Dalang.84,ElKarouiLepeltierMillet.92, Ghoussoub.83}: the constraint given by the fixed embedding target~$\nu$ is not compatible with the classical results and techniques. 
In Theorem~\ref{th:MongeDense} we establish the more general result that $\cT(\nu)\subseteq \cR(\nu)$ is dense for weak convergence. Our proof is constructive and gives insight into why the first result can fail when $G$ is irregular. In a nutshell, the idea is to use a short initial segment of the Brownian path as a randomization device for the rest of the problem. In fact, we show that for randomized stopping times $\xi\in\cR(\nu)$ and reward functions $G$ that do not depend on an initial segment of the paths (in a sense to be made precise), the expectation of $G$ stopped at $\xi$ can be exactly replicated by a stopping time $\tau\in\cT(\nu)$, without any need for approximation (cf.\ Proposition~\ref{pr:RandomGenerator}). These ideas seem to be novel in the literature.

The optimal Skorokhod embedding problem over $\cR(\nu)$ is a linear programming problem with constraints and thus has a dual programming problem. Formally, the domain of the dual problem is the set of all pairs $(M,\psi)$ where $M$ is a martingale with $M_{0}=0$ and $\psi:\R\to\R$ is a function such that $M_{t}+\psi(B_{t})\geq G_{t}$ for  $t\geq0$. The dual problem then consists in minimizing $\nu(\psi):=\int \psi \,d\nu$ over all such pairs $(M,\psi)$. More specifically, \cite{BeiglbockCoxHuesmann.14} uses martingales $M$ that satisfy a quadratic growth condition relative to $B$ and functions $\psi$ that are continuous and satisfy a growth condition, or \cite{GuoTanTouzi.15a} works with similar functions $\psi$ and supermartingales $M$.
Such a dual problem has been used in numerous examples to help determine specific optimal Skorokhod embeddings; e.g., \cite{CoxKinsley.18,CoxObloj.11, CoxWang.13, HenryLabordereOblojSpoidaTouzi.12, Hobson.98, HobsonKlimmek.12, HobsonKlimmek.15}. Moreover, it plays a vital role in~\cite{BeiglbockCoxHuesmann.14} (see also \cite{GuoTanTouzi.15}) in deriving a general monotonicity principle that describes the barriers representing optimal Skorokhod embeddings through their hitting times. While dual solutions have been found in those specific examples, it has been observed in~\cite{BeiglbockCoxHuesmann.14} that the dual problem fails to have a solution in general; that is, the minimum is not attained. We refer to the survey~\cite{Hobson.11} for further references.

The second part of this paper introduces a novel relaxation of the dual problem and establishes the existence of its solution as well as the absence of a duality gap. Both of these results are obtained without continuity conditions for $G$, thus paralleling the generality of Kellerer's theorem~\cite{Kellerer.84} in optimal transport and improving the results on the absence of a duality gap in~\cite{BeiglbockCoxHuesmann.14,GuoTanTouzi.15a}. 
There are no previous existence results in our setting. We can mention \cite{GhoussoubKimPalmer.18} for a PDE approach with attainment in a different dual problem. 
Here the reward $G$ is given by the integral of a continuous, finite-dimensional Lagrangian with exponential decay and the marginals are absolutely continuous with compact support.
Remarkably, \cite{GhoussoubKimPalmer.18} allows for multidimensional Brownian motion.

Even if $G$ is continuous and bounded, our main issue is the lack of compactness for the martingale component. Broadly speaking, for a given minimizing sequence $(M^{n},\psi^{n})$ the functions $\psi^{n}$ may have large positive values, so that the inequality $M^{n}_{t}+\psi^{n}(B_{t})\geq -\|G^{-}\|_{\infty}$ does not immediately result in a lower bound for $M^{n}$. On the other hand, the limit of a sequence of continuous-time (super-)martingales may fail to be a supermartingale in the absence of a lower bound. A crucial feature of our relaxation is to work with local martingales that have uniform lower bounds on sets where the Brownian motion is bounded and bounded away from the extremes of the support of~$\nu$. It turns out that on such sets we can obtain enough compactness, while preserving the ``weak'' side of the duality. Less surprisingly, our functions $\psi$ are merely in $L^{1}(\nu)$ rather than continuous. We provide counterexamples showing that our positive results on duality can fail if one were to insist on true (super-)martingales or continuous functions; cf.~Section~\ref{se:counterex}.

An analogous duality result was obtained in~\cite{BeiglbockNutzTouzi.15} (see also \cite{BeiglbockLimObloj.17}) for the so-called martingale optimal transport problem in a single period. The main compactness issue sketched above does not arise in that setting; basically, limits of martingales remain martingales in the discrete-time setting. A softer part of our proof does overlap with~\cite{BeiglbockNutzTouzi.15}, namely the use of capacity theory to generalize from continuous to measurable reward functionals.
Martingale optimal transport refers to a transport problem where the transports are constrained to be martingales. Originally motivated by model uncertainty in financial applications~\cite{Hobson.98}, 
a rich literature has emerged around this subject; see \cite{Hobson.11, Obloj.04, Touzi.14} for surveys and, e.g., \cite{AcciaioBeiglbockPenknerSchachermayer.12, BeiglbockHenryLaborderePenkner.11, BouchardNutz.13, BeiglbockHenryLabordereTouzi.15, BurzoniFrittelliMaggis.15, CheriditoKupperTangpi.14, DeMarcoHenryLabordere.15, FahimHuang.14, GhoussoubKimLim.15, GozlanRobertoSamsonTetali.14, HuesmannStebegg.18, Nutz.13, NutzStebegg.16, NutzStebeggTan.17, Zaev.14} for models in discrete time and \cite{BiaginiBouchardKardarasNutz.14, CoxHouObloj.14, CoxObloj.11, DolinskySoner.12, DolinskySoner.14, GalichonHenryLabordereTouzi.11, HenryLabordereOblojSpoidaTouzi.12, HenryLabordereTanTouzi.14, HirschProfetaRoynetteYor.11, Hobson.15, KallbladTanTouzi.15, NeufeldNutz.12, Nutz.14, Stebegg.14, BeiglbockHuesmannStebegg.15, TanTouzi.11} for continuous time. Many of these works exploit connections to the Skorokhod embedding problem. In particular, the optimal Skorokhod embedding problem can be related, by a time change, to the continuous-time martingale transport between two marginals. While a duality theory with dual existence has been elusive for the latter transport problem, the arguments in the present paper are expected to lead to such a result since the compactness issues are similar. This will be investigated in future work.

Our results on duality allow us to derive a monotonicity principle in the spirit of~\cite{BeiglbockNutzTouzi.15}, under an integrability condition on~$G$. Namely, there exists a universal support $\Gamma$ that characterizes all optimal embeddings: $\xi\in\cR(\nu)$ is optimal if and only if $\xi(\Gamma)=1$. 
This complements the monotonicity principle of~\cite{BeiglbockCoxHuesmann.14} which gives a geometric condition on the support that is necessary for optimality, but not sufficient. By contrast, our result yields a necessary and sufficient condition. However, the geometry of the support is merely described in a weaker form, through the construction of $\Gamma$ as the set where a dual optimizer equals the reward functional (we exemplify in Section~\ref{se:cave} how this can be utilized to obtain more specific geometric statements). It is an interesting question for future research how to unify these results, though the answer does not seem to be within reach with present concepts and knowledge.
It is noteworthy that the integrability condition is crucial for any monotonicity principle to hold; indeed, we provide a surprising example showing that for general~$G$, the optimality of embeddings cannot be characterized in terms of the support. This contrasts with cyclical monotonicity properties in classical optimal transport~\cite{Villani.09} as well as the result of~\cite{BeiglbockNutzTouzi.15} which suggest that optimality of transports can be characterized by their geometry in great generality.

In Section~\ref{se:cave} we illustrate how our result on dual existence can be exceptionally useful to describe the geometry of optimal Skorokhod embeddings in a concrete case. Namely, we specialize to a particular class of convex-concave reward functions $G_{t}=g(t)$ which give rise to embeddings that can be represented as hitting times of sets consisting of both a left and a right barrier in the $(t,x)$-plane. This class of ``cave'' embeddings with a double boundary, unifying the ones of Root and Rost, was introduced in~\cite{BeiglbockCoxHuesmann.14}. In contrast to Root's and Rost's, such barriers are not determined by~$\nu$ alone but depend on the details of the function $g$, and the arguments in~\cite{BeiglbockCoxHuesmann.14} do not lead to a characterization of the optimal barriers. We provide such a characterization through a variational condition, very much inspired by \cite{CoxKinsley.18} which studies a different class of reward functions. The condition is related to the principle of smooth fit for free boundary problems, and thus it is no surprise that the proof of sufficiency for optimality takes the form of a verification argument. In~\cite{CoxKinsley.18}, the proof of necessity is an impressive tour de force through a discretization that is carried out in the separate paper~\cite{CoxKinsley.18b}. On the strength of our result on dual existence, we can provide a much more direct proof. First, we show that when the reward function is Markovian (that is, a deterministic function of time and current state), the abstract martingale $M$ in the dual can be replaced by a function of two variables. Then, we can apply relatively soft probabilistic arguments to derive the variational characterization. Importantly, our proof reveals that similar conditions should extend to much more general classes of embeddings, and also suggests that regularity results for the stopping boundaries can be obtained through the dual maximizer. These aspects will be investigated in separate work.

The remainder of this paper is organized as follows. Section~\ref{se:primalProblem} details the optimal Skorokhod embedding problem and states the equality of the formulations with randomized and non-randomized stopping times for regular reward functions. The proof is given in Section~\ref{se:approximation} where it is shown more generally how randomized stopping times can be approximated by non-randomized ones. Section~\ref{se:dualProblem} introduces the relaxed dual problem and provides the existence of a solution. The absence of a duality gap is proved in Section~\ref{se:duality} where we also state the monotonicity principle. In Section~\ref{se:cave} we discuss cave embeddings and characterize the optimal barriers by exploiting our abstract results. Counterexamples regarding the formulation of the dual problem and the monotonicity principle are gathered in Section~\ref{se:counterex}. For simplicity of exposition, we use a second moment condition on $\nu$ throughout the paper. Appendix~\ref{se:firstMoment} explains how this can be replaced by a finite first moment without much effort.

\section{The Primal Problem}\label{se:primalProblem}

Let $C_{0}(I)$ be the set of continuous real-valued functions $\omega=(\omega_{t})_{t\in I}$ with $\omega_{0}=0$, for any interval $0\in I \subseteq \R$. We denote by $S$ the space of stopped continuous paths; that is, pairs $(\omega,t)$ with $t\in\R_{+}$ and $\omega\in C_{0}([0,t])$. The set $S$ is a Polish space under the topology induced by the metric
\[
  d((f,s),(g,t))= |t-s| \vee \sup_{u\geq0}|f_{u\wedge s}-g_{u\wedge t}|.
\]
We note that any $(\omega,t)\in C_{0}(\R_{+})\times\R_{+}$ projects to an element $(\omega|_{[0,t]},t)$ of~$S$. Conversely, we can embed $S$ in $C_{0}(\R_{+})\times\R_{+}$, say by continuing any stopped path in a constant fashion. This identification will sometimes be used implicitly.

We fix a reward function $G: S\to [0,\infty]$. Any such $G$ can be seen as an ``adapted'' process (in the sense of Galmarino's test) on the canonical space~$C_{0}(\R_{+})$ in that $G_{t}(\omega):= G(\omega,t)$ depends only on $\omega|_{[0,t]}$. If $G$ is Borel measurable, it can be identified with an optional process on $C_{0}(\R_{+})$ when the latter is equipped with the (raw) canonical filtration $\F=(\cF_{t})_{t\geq0}$ generated by the coordinate-mapping process $B$; i.e., $\cF_{t}=\sigma(B_{s},\,s\leq t)$ where $B_{t}(\omega)=\omega_{t}$ for $(\omega,t)\in C_{0}(\R_{+})\times\R_{+}$. Conversely, any adapted (optional) process induces a function (Borel function) on $S$. We recall that with respect to $\F$, any measurable adapted process is already optional (and even predictable) and refer to \cite[Nos.\ IV.94--103, pp.\,145--152]{DellacherieMeyer.78} or \cite[Section~3]{BeiglbockCoxHuesmann.14} for further background on path spaces and their filtrations. 

We equip $C_{0}(\R_{+})$ with the Wiener measure $\W$ so that $B$ is a standard Brownian motion with initial distribution $B_{0}\sim\delta_{0}$. In what follows, probabilistic notions generally refer to the Wiener measure and the canonical filtration unless a different context is given. All metric spaces are equipped with their Borel $\sigma$-fields,  (in)equalities of processes are to be understood up to evanescence (meaning that the projection of the exceptional set is $\W$-null) and (in)equalities of random variables are in the almost-sure sense.

Consider a centered distribution $\nu$ on $\R$ with finite second moment (but see Appendix~\ref{se:firstMoment} for a generalization to finite first moment). Let~$\cT$ be the set of all a.s.\ finite $\F$-stopping times and let $\cT(\nu)$ be the subset of all $\tau\in\cT$ such that $E[\tau]<\infty$ and $B_{\tau}\sim \nu$. The set $\cT(\nu)$ is nonempty; a number of classical embeddings $\tau\in\cT(\nu)$ can be found in~\cite{Obloj.04}. The optimal Skorokhod embedding problem with respect to stopping times is
\[
  \bS_{\cT}(G)=\sup_{\tau\in\cT(\nu)} E[G_{\tau}].
\]
Here and below, outer integrals are used whenever the integrand is not measurable.
The optimal Skorokhod embedding problem is often discussed with respect to randomized stopping times, defined as follows.

\begin{definition}\label{de:RST}
A probability measure $\xi$ on $C_0(\R_+) \times \R_+$ with disintegration $\xi(d\omega,dt) = \W(d\omega)\xi_\omega(dt)$ is a \emph{randomized stopping time} if $\xi_\omega(\R_+)=1$ for almost all $\omega$ and $\omega\mapsto \xi_\omega([0,t])$ is $\cF_t$-measurable for all $t\geq0$.
We denote the set of randomized stopping times by $\cR$.
\end{definition}

We emphasize that our randomized stopping times are defined to stop in finite time. We can embed $\cT$ in $\cR$ in a canonical way: $\tau$ is mapped to the randomized stopping time $\xi^{\tau}$ with kernel $\xi^{\tau}_\omega := \delta_{\tau(\omega)}$. The image of $\cT$ under this embedding is denoted by $\cR_{\cT}$; we will refer to its elements as \emph{non-randomized stopping times}. We note the analogy between $\cR$ and Kantorovich transports on the one hand versus $\cR_{\cT}$ and Monge transports on the other.

\begin{definition}\label{de:Rembedding}
  The set $\cR(\nu)$ consists of all randomized stopping times $\xi\in\cR$ such that $\xi(\bt)<\infty$ and $\xi\circ B^{-1}=\nu$. We write $\cR_{\cT}(\nu)$ for the subset of non-randomized stopping times.
\end{definition}

Here $\bt$ is the projection given by $\bt(\omega,t)=t$ and the two conditions correspond to a finite first moment and the marginal constraint $\nu$. In particular, if $\xi=\xi^{\tau}$ represents a stopping time, then $\xi(\bt)=E[\tau]$ and $\xi\circ B^{-1}$ is the law of $B_{\tau}$.
The optimal Skorokhod embedding problem is then given by
\[
  \bS(G)=\sup_{\xi\in\cR(\nu)} \xi(G)
\]
where again $\xi(G):=\int G\,d\xi$. More briefly, we will also call this the \emph{primal problem}. 

It is an obvious question, not previously addressed in the literature, to give a general condition under which the two formulations of the optimal Skorokhod embedding problem have the same value. The answer we provide is:  whenever the reward function $G$ has sufficiently regular paths.

\begin{theorem}\label{th:MongeValue}
Let $G: C_0(\R_+) \times \R_+ \to [0,\infty)$ be Borel and adapted, and let $t\mapsto G_{t}(\omega)$ be lower semicontinuous for all $\omega\in C_0(\R_+)$. Then
\[
  \sup_{\xi\in\cR(\nu)} \xi(G) = \sup_{\tau\in\cT(\nu)} E[G_{\tau}].
\]
\end{theorem}

The proof is stated in the next section; it is a consequence of Theorem~\ref{th:MongeDense} below and the proof of the latter will also help to understand where and how the regularity of $G$ comes into play. Theorem~\ref{th:MongeValue} should be contrasted with results in unconstrained optimal stopping where value functions over stopping times and randomized stopping times agree for general measurable reward functions; see e.g.\ \cite[Theorem~2.1]{GyongySiska.08}. In particular, the following (well-known) example shows that lower semicontinuity of $G$ is an important assumption in Theorem~\ref{th:MongeValue}: the nonstandard constraint given by the marginal $\nu$ markedly changes the character of the question, as will be discussed in more detail in Section~\ref{se:approximation}.

\begin{example}\label{ex:counterexApprox}
  Let $\nu$ have an atom of mass $a\in(0,1)$ at the origin and let $G$ be the bounded upper semicontinuous function $G(\omega,s)=\1_{\{0\}}(s)$. Then
  \[
    \sup_{\xi\in\cR(\nu)} \xi(G)=a > 0 = \sup_{\tau\in\cT(\nu)} E[G_{\tau}].
  \]
  Indeed, any $\tau\in\cT$ satisfies $\W\{\tau=0\}\in\{0,1\}$ by the Blumenthal 0-1 law. Hence, any $\tau\in\cT(\nu)$ must be strictly positive a.s.\ which entails that $E[G_{\tau}]=0$. On the other hand, we can find $\xi\in\cR(\nu)$ with $\xi(\{0\})=a$ and any such $\xi$ attains the supremum.
\end{example}

\section{Approximation of Randomized Stopping Times with Fixed Marginal}\label{se:approximation}

The main aim of this section is a density result with respect to the weak topology on $C_0(\R_+) \times \R_+$ (induced by the continuous bounded functions on that space, as usual). We recall that  $\cR_{\cT}(\nu)$ is the embedding of $\cT(\nu)$ in $\cR$.

\begin{theorem}\label{th:MongeDense}
Let $\nu$ be a centered probability on $\R$ with finite second moment. Then $\cR_{\cT}(\nu)\subseteq\cR(\nu)$ is dense for the topology of weak convergence.
\end{theorem}

Again, this should be compared with classical results on the convergence of (convex combinations of) stopping times to randomized ones, such as \cite{BaxterChacon.77, Dalang.84,ElKarouiLepeltierMillet.92, Ghoussoub.83}. Such approximations may not, in general, respect the constraint given by the marginal~$\nu$. In fact, in contrast to the unconstrained setting, the extreme points of $\cR(\nu)$ are not necessarily contained in $\cR_{\cT}(\nu)$ (a counterexample can be deduced from \cite[Example~6.19]{BeiglbockCoxHuesmann.14}).

\begin{proof}[Proof of Theorem~\ref{th:MongeValue}.]
  When $G$ is bounded, the claim is a direct consequence of Theorem~\ref{th:MongeDense} and the characterization of weak convergence in~\cite[Corollary~2.9 and Proposition~2.11]{JacodMemin.81}. (The main point of~\cite{JacodMemin.81} is that since we are dealing with measures that all have the same first marginal~$\W$, weak convergence implies convergence under bounded test functions that are continuous in~$t$ but merely measurable in~$\omega$.) The result for general $G$ now follows by monotone approximation.
\end{proof}

Our proof of Theorem~\ref{th:MongeDense}, presented in the remainder of this section, is constructive and gives direct insight why a singularity at the origin (as in Example~\ref{ex:counterexApprox}) is an obstruction. One result to be established as part of the proof is that for randomized stopping times $\xi$ and reward functions $G$ that do not depend on an initial segment of the paths (in a sense to be made precise), the expectation $\xi(G)$ can be exactly replicated by a stopping time $\tau\in\cT(\nu)$, without any need for approximation (cf.\ Proposition~\ref{pr:RandomGenerator} below).

\subsection{Proof of Theorem~\ref{th:MongeDense}}

Our first aim is to formalize and show that any $\xi\in\cR(\nu)$ can be approximated by randomized stopping times that do not stop right after time 0.
We denote by $|\cdot|$ the Euclidean norm in any dimension.

\begin{definition}
Let $\eta>0$ and $\tau_\eta = \inf\{t :\,|(t,\omega_{t})| \geq \eta\}$. We say that a randomized stopping time $\xi$ is \emph{bounded away from $0$ with lower bound $\eta>0$} if
\[\xi_\omega([\tau_\eta(\omega),\infty)) = 1 \quad \text{for almost all} \quad \omega \in C_0(\R_+).\]
The set of all such $\xi$ is denoted $\cR^\eta$. 
We also set $\cR^+=\cup_{\eta>0}\cR^\eta$; any $\xi\in\cR^+$ is called \emph{bounded away from $0$}. Finally, $\cR^\eta(\nu):=\cR^{\eta}\cap\cR(\nu)$ and $\cR^+(\nu):=\cR^+\cap\cR(\nu)$.
\end{definition}

The terminology of lower bound is convenient but slightly abusive: intuitively, $\xi\in\cR^{\eta}$ above means that $\xi$ happens after time~$\tau_{\eta}$ (but the time $\tau_{\eta}$ is not bounded away from zero in the usual sense). Some more notation will be useful.

\begin{definition}\label{de:gluedAndShift}
Let $\omega,\omega' \in C_0(\R_+)$, $t \in \R_+$, $\tau \in \cT$ and $\xi \in \cR$.

\begin{enumerate}
\item The path of $\omega$ and $\omega'$ \emph{glued at time $t$} is
\[(\omega \oplus_t \omega')(s) := \omega(s \wedge t) + \omega'(s - t \vee 0),\quad s\geq0.\]

\item The path of $\omega$ \emph{after time $t$} is
\[ \omega^{t\mapsto}(s) := \omega(t+s) - \omega(t),\quad s\geq0.\]

\item  The randomized stopping time \emph{$\xi$ shifted by $\tau$}, denoted $\tau \oplus \xi$, is defined by  its kernel
\[(\tau \oplus \xi)_\omega([0,t]) := \1_{\tau(\omega) < t}\xi_{\omega^{\tau(\omega)\mapsto}}([0,t - \tau(\omega)]),\quad t\geq0.\]
\end{enumerate}
\end{definition}

The definition in (iii) can be understood as the randomized stopping time $\xi$ applied to the Brownian motion started at $(\tau,B_{\tau})$. For instance, if $\tau=t_{0}$ is deterministic and $\xi=\xi^{\sigma}$ corresponds to a nonrandomized stopping time $\sigma>0$, then $\tau \oplus \xi$ corresponds to the stopping time $(t_{0}\oplus\sigma)(\omega) = t_{0} + \sigma(\omega^{t_{0}\mapsto}).$

\begin{lemma}\label{le:boundProps}
For $0 < \eta < 1$, define
\begin{align*}
\rho_\eta'(\omega) := \inf \{ t: |\omega_t| = \eta\}, \qquad \rho_\eta(\omega) := \inf\{t \geq \rho_\eta'(\omega) : |\omega_t| \in \{0,\sqrt{\eta}\}\}.
\end{align*}
For almost all $\omega \in C_0(\R_+)$, we have 
\begin{enumerate}
\item $\rho_\eta(\omega) \to 0$ as $\eta \to 0$,
\item $d((\omega^{\rho_\eta(\omega)\mapsto},t),(\omega,t)) \to 0$ as $\eta \to 0$, for all $t\geq 0$.\end{enumerate}
\end{lemma}

\begin{proof}
We show that (i), (ii) hold on the set $I$ of all paths $\omega\in C_0(\R_+)$ that are not initially constant; i.e., 
$\omega|_{[0,\eps]}\not\equiv0$ for all $\eps>0$. Notice that $\W(I)=1$.

(i) If $\omega\in C_0(\R_+)$ is such that $\rho_\eta(\omega)$ does not converge to $0$, we can
find $\eps > 0$ and a sequence $\eta_n \to 0$ such that $\rho_{\eta_n}(\omega) \geq \epsilon$ for all $n$. In particular, this implies that
$\sup_{s \leq \epsilon} |\omega(s)| \leq \sqrt{\eta_n}$ for all $n$ and therefore $\omega|_{[0,\eps]}\equiv0$; that is, $\omega\notin I$.

(ii) Fix $\omega \in I$ and $t>0$. Let $\eta > 0$ be small enough so that $\rho_\eta(\omega) < t$ and
consider some $0 < s < t$. For $s \leq \rho_\eta(\omega)$ we clearly have
\[|\omega^{\rho_\eta(\omega)\mapsto}(s) - \omega(s)| \leq \sup_{u \leq 2\rho_\eta(\omega)} 3 | \omega(u)|.\]
Whereas for $\rho_\eta(\omega) < s \leq t$ we have
\[|\omega^{\rho_\eta(\omega)\mapsto}(s) - \omega(s)| \leq \sqrt{\eta} + \sup_{s \leq t}|\omega(\rho_\eta(\omega)+s) - \omega(s)|.\]
Combining these two inequalities, we obtain that
\[ \sup_{s \leq t} |\omega^{\rho_\eta(\omega)\mapsto}(s) - \omega(s)| \leq \sup_{u \leq 2\rho_\eta(\omega)} 3 | \omega(u)| +\sqrt{\eta}+ \sup_{s \leq t}|\omega(\rho_\eta(\omega)+s) - \omega(s)|.\]
Since $\omega$ is uniformly continuous on compact intervals, the right-hand side tends to $0$ when $\rho_\eta(\omega)\to0$, and the latter holds by~(i) as $\eta\to0$.
\end{proof}

\begin{remark}\label{rk:weakTopAdapted}
  If $G\in C_{b}(C_0(\R_+) \times \R_+)$ and $\xi\in\cR$, then $\xi(G)=\xi(G(B_{\cdot\wedge\bt},\bt))$ by the adaptedness property of $\xi\in\cR$. As a result, the weak convergence on~$\cR$ is also induced by the subset of adapted functions $G$ which, in turn, can be identified with $C_{b}(S)$.
\end{remark}

We can now show that any $\xi\in\cR(\nu)$ can be approximated with embeddings that are bounded away from~0---except in the trivial case $\nu=\delta_{0}$ where the assertion clearly fails.

\begin{proposition}\label{pr:RplusDense}
Let $\nu \neq \delta_0$. Then $\cR^{+}(\nu)$ is weakly dense in $\cR(\nu)$.
\end{proposition}

\begin{proof}
As $\nu \neq \delta_0$, its potential function $x\mapsto u_\nu(x):= \int |x-y| \, {\nu}(dy)$
satisfies $u_\nu(0) > 0$, and then by continuity of $u_\nu$ we even have
$\min_{y \in [-\sqrt{\eta},\sqrt{\eta}]} u_{\nu}(y)\geq \sqrt{\eta}$ for $\eta > 0$ small enough. This shows that
\[\mu_\eta := \frac{1}{2}(\delta_{-\sqrt{\eta}} + \delta_{\sqrt{\eta}}) \leq_c \nu\]
where $\leq_c$ denotes the convex order.
 As a result, for $\eta > 0$ small enough we can embed $\nu$ in a Brownian motion with initial distribution $\mu_\eta$ by a stopping time $\chi_\eta$ with $E[\chi_\eta]<\infty$. (See~\cite{Obloj.04} for the relevant background.)

Given $\xi\in\cR(\nu)$, let $\rho_\eta$ be as in Lemma \ref{le:boundProps} and define a family of randomized stopping times $\xi^\eta \in \cR$ via their
disintegration,
\[ \xi^\eta_\omega := \1_{\{\omega(\rho_\eta(\omega)) = 0\}}(\rho_\eta \oplus \xi)_\omega + \1_{\{|\omega(\rho_\eta(\omega))| = \sqrt{\eta}\}}(\rho_\eta \oplus \chi_\eta)_\omega.\]
Then we have $\xi^{\eta}\in\cR^{\eta}(\nu)$ by construction.
Next, we show that $\xi^\eta \to \xi$ weakly as $\eta \to 0$.
Let $G : C_0(\R_+) \times \R_+ \to \R$ be bounded, continuous and adapted, then
\begin{align*}
\xi(G) &= \int_{C_0(\R_+)} \int_{\R_+} G(\omega,s) \xi_\omega(ds) \W(d\omega) \\ 
&= \int_{C_0(\R_+)}\left[\int_{C_0(\R_+)} \int_{\R_+} G(\omega,s) \xi_\omega(ds) \W(d\omega)\right] \W(d\omega').
\end{align*}
Using the stationarity and independence of Brownian increments, we can similarly write the expectation $\xi^\eta(G)$ as
\begin{align*}
 \int_{C_0(\R_+)}
\left[ \1_{\omega'(\rho_\eta(\omega')) = 0} \int_{C_0(\R_+)}\int_{\R_+} G(\omega' \oplus_{\rho_\eta(\omega')} \omega,s + \rho_\eta(\omega'))  \xi_\omega(ds)\W(d\omega) \right.\\
\left.+\1_{\omega'(\rho_\eta(\omega')) \neq 0} \int_{C_0(\R_+) \times \R_+} G(\omega' \oplus_{\rho_\eta(\omega')} \omega,s + \rho_\eta(\omega')) \chi_\eta(d\omega,ds)\right] \W(d\omega').
\end{align*}
Therefore,
\begin{align*}
&|\xi(G) - \xi^\eta(G)| \\
&\leq \int_{C_0(\R_+)}\left[\int_{C_0(\R_+)} \int_{\R_+} |G(\omega,s) -  G(\omega' \oplus_{\rho_\eta(\omega')} \omega,s + \rho_\eta(\omega')) | \times \right.\\
& \hspace{7em} \xi_\omega(ds) \W(d\omega) \Bigg]\W(d\omega')
+ 2\|G\|_{\infty} \W\{\omega':\,\omega'(\rho_\eta(\omega')) \neq 0\}.
\end{align*}
We note that $\W\{\omega':\,\omega'(\rho_\eta(\omega')) \neq 0\} = \sqrt{\eta}\to0$ as a consequence of the martingale property.

On the other hand, Lemma~\ref{le:boundProps} yields that $\rho_\eta(\omega') \to 0$ for all $\omega'$ outside a nullset, and for such $\omega'$ the second part of that lemma 
and the continuity of $G$ imply that
\[
  |G(\omega,s) -  G(\omega' \oplus_{\rho_\eta(\omega')} \omega,s + \rho_\eta(\omega')) | \to 0
\]
for all $\omega$. By bounded convergence, we conclude that $|\xi^\eta(G) - \xi(G)| \to 0$ for $\eta \to 0$ as desired. 
\end{proof}

For brevity, let us write $C_b$ for the set of bounded continuous functions on $C_0(\R_+) \times \R_+$. 
Next, we introduce the subset of test functions which are independent of the initial segment of the path, or more precisely, their value depends on the path $\omega$ before time $\tau_{\eta}$ only through the level $\omega_{\tau_{\eta}(\omega)}$ at that time.

\begin{definition}
Given $\eta>0$, we write $C_b^\eta$ for the set of all $G\in C_b$ with the property that if $ \omega, \omega'\in C_{0}(\R_{+})$ satisfy 
\[
  \tau_\eta(\omega)=\tau_\eta(\omega')=:t_0\mbox{ and }\omega|_{[t_0,\infty)} = \omega'|_{[t_0,\infty)}, \quad \mbox{then}\quad
  G(\omega, t)= G(\omega', t)
\]
for all $t\geq0$. We set $C_b^+:= \cup_{\eta>0} C_b^\eta$.
\end{definition}

This definition is chosen so that the monotonicity property
\begin{equation}\label{eq:CetaMonotone}
C_{b}^{\eta}\supseteq C_{b}^{\eta'}\quad\mbox{for}\quad\eta\leq\eta'
\end{equation}
holds. In particular, $C_b^+$ is the increasing limit of $C_b^\eta$ as $\eta\to0$.

\begin{lemma}\label{le:EasyTesting}
The weak topology on $\cR(\nu)$ induced by $C_{b}^{+}$ coincides with the usual weak topology.
\end{lemma}

\begin{proof}
Let $G\in C_b$; then $G_{k}(\omega,t):=G(\omega^{\tau_{1/k}(\omega)\mapsto},t)$ defines a function belonging to $C_{b}^{1/k}\subseteq C_{b}^{+}$. We have $d((\omega^{\tau_{1/k}(\omega)\mapsto},t),(\omega,t))\to0$ similarly as in Lemma~\ref{le:boundProps} and as $(G_{k})$ is uniformly bounded this implies that
 $\xi(G_{k})\to \xi(G)$ for all $\xi\in\cR(\nu)$. As a result, $C_{b}^{+}$ separates the points of $\cR(\nu)$; i.e., the induced weak topology $\mathcal{T}^{+}$ is Hausdorff. Since $\cR(\nu)$ is compact in the usual weak topology~$\mathcal{T}$ (see for instance \cite[Theorem~3.14]{BeiglbockCoxHuesmann.14}) and $\mathcal{T}^{+}\subseteq\mathcal{T}$, this already implies the result. Indeed, the identity map from $\mathcal{T}$ to $\mathcal{T}^{+}$ is continuous and maps compacts to compacts, hence is a homeomorphism.
\end{proof}

A key insight is that for $G\in C_b^+$, the expectation under a randomized stopping time that is suitably bounded away from zero is (exactly) replicated by a non-randomized stopping time. In particular, \emph{the necessity for approximation when $G\in C_b$ can be attributed to the initial portion of the paths.}

\begin{proposition}\label{pr:RandomGenerator} 
  Let $\eta>0$ and $\xi\in \cR^\eta(\nu)$. Then there exists a non-randomized stopping time $\bar\xi\in \cR^\eta(\nu)$ such that $\bar\xi(G)= \xi(G)$ for all $G\in C_b^\eta$.
\end{proposition}

Before stating the proof, let us show how to combine the above results to conclude Theorem~\ref{th:MongeDense}.

\begin{proof}[Proof of Theorem~\ref{th:MongeDense}]
The case $\nu=\delta_{0}$ is trivial; we assume that $\nu\neq\delta_{0}$. Fix $\xi\in \cR(\nu)$, then by Proposition~\ref{pr:RplusDense} we can find $\xi^{n}\in \cR^{1/n}(\nu)$ such that $\xi^{n}\to \xi$. Next, we use Proposition \ref{pr:RandomGenerator} to find corresponding non-randomized stopping times $\bar\xi^{n}\in \cR^{1/n}(\nu)$. Let $G\in C_{b}^{+}$, then $G\in C_{b}^{\eta}$ for some $\eta>0$ and for all $n\geq 1/\eta$ we have by~\eqref{eq:CetaMonotone} that $\bar\xi^{n}(G)=\xi^{n}(G)$. In particular, Lemma~\ref{le:EasyTesting} implies that $(\bar\xi^{n})$ and $(\xi^{n})$ have the same weak limit; that is,  $\bar\xi^{n}\to \xi$.
\end{proof}

\subsection{Proof of Proposition~\ref{pr:RandomGenerator}}

The basic idea for this proof is to use the initial segments of the paths as a randomization device which, since we are only testing with $G\in C_b^\eta$, does not affect the evaluation of $G$. 
We first state two auxiliary results.

Let $\lambda$ be the Lebesgue measure on $[0,1]$. We consider the product space $\bar C_{0}(\R_{+}):=C_{0}(\R_{+})\times[0,1]$ equipped with the product $\sigma$-field $\bar\cF=\cF\otimes\cB([0,1])$, the product measure $\bar\W=\W\otimes\lambda$ and the product filtration $\bar\F$. Let $\bar B$ be the process defined by $t\mapsto\bar B_t(\omega,u)=\omega_t$ for $(\omega,u)\in\bar C_{0}(\R_{+})$; note that $\bar B$ is a Brownian motion under $\bar\W$. 
The following is well known (see for instance the proof of \cite[Theorem~3.8]{BeiglbockCoxHuesmann.14} and the subsequent lemma).

\begin{lemma}\label{le:rho}
  Let $\xi\in \cR$ have disintegration $\xi=\W(d\omega)\xi_{\omega}(ds)$. There exists an $\bar\F$-stopping time $\rho$ such that
  \begin{equation}
    \label{eq:rhodefn}
     \rho(\omega,u) =\inf \{ t \ge 0 : \xi_\omega([0,t]) \ge u\}\quad \mbox{for a.e.} \quad (\omega,u)\in \bar C_{0}(\R_{+})
  \end{equation}
  and $\bar\W\circ (\bar B,\rho)^{-1}=\xi$; that is, $E^{\bar\W}[G(\bar B,\rho)]=\xi(G)$ for every $G\in C_{b}$.
\end{lemma}

The second auxiliary result concerns the ``internal'' randomization device that will be used in lieu of the external randomization $([0,1],\lambda)$ as in the preceding lemma. This is somewhat involved because the randomization needs to be implemented conditionally on the level $B_{\tau^\eta}(\omega)$---indeed, $G\in C_{b}^{\eta}$ is independent of how a path reached this level at time $\tau^\eta(\omega)$, but of course not of the level itself.

Fix $\eta\geq0$. We note that the level $h(\omega):=B_{\tau^\eta}(\omega)$ of a path $\omega$ at the (almost surely finite) time $\tau_\eta$  satisfies $h(\omega)\in (-\sqrt{\eta}, \sqrt{\eta})$. Moreover, 
\[
\tau_\eta(\omega)=\sqrt{\eta^{2}-h(\omega)^{2}}
\]
depends on $\omega$ only through $h=h(\omega)$. Given $h\in\R$,  we introduce the set
\[C_h := \{ \omega \in C_0(\R_+) :\, \tau_\eta(\omega)<\infty,\;B_{\tau^\eta}(\omega) = h \};\]
we think of $C_{h}$ as a set of initial segments of paths (since the path after $\tau_\eta$ will not be used).
Given $f \in C_h$ and $\omega \in C_0(\R_+)$, we set
\[f \oplus \omega:=f \oplus_{\tau_{\eta}(f)} \omega\]
for brevity.
We also denote by $\W_h$ the conditional law of $B$ given $B_{\tau^\eta}=h$. That is, $\W_{h}$ is a stochastic kernel on $(-\sqrt{\eta},\sqrt{\eta})\times C_{0}(\R_{+})$ such that 
\[
  \W[A | B_{\tau^\eta} = h]=\W_h(A) = \W_{h}(A\cap C_{h})
\]
for $h\in (-\sqrt{\eta},\sqrt{\eta})$ and $A\in\cB(C_{0}(\R_{+}))$. In particular,
\[
\W(A) = \int_{\R} \W_h(A)\,\mu(dh),\quad \mu:=\Law(B_{\tau^\eta}).
\]

\begin{lemma}\label{le:WhAtomless}
  The measure $\W_{h}$ is atomless for all $h\in(-\sqrt{\eta},\sqrt{\eta})$.
\end{lemma}

\begin{proof}
  Consider the map $\Phi: C_{0}(\R_{+})\to C_{0}(\R_{+})$ given by $\Phi(\omega)=\omega^{\tau_{\eta}(\omega)\mapsto}$. By the stationarity and independence of Brownian increments, the pushforward $\W_{h}\circ\Phi^{-1}$ is the Wiener measure and in particular atomless. As a consequence, $\W_{h}$ is atomless, for if $\W_{h}$ had an atom then any pushforward would also have an atom. 
\end{proof}

\begin{lemma}\label{le:maptoLebesgue}
  Let $(P_{y}(dz))$ be a stochastic kernel $(Y,\cY)\times (Z,\cB(Z))\to [0,1]$ where $Z$ is a Polish space and $(Y,\cY)$ is a measurable space. If $P_{y}$ is atomless for all $y\in Y$, there exists a jointly measurable map $(y,z)\mapsto\phi_{y}(z)\in[0,1]$ such that $P_{y}\circ \phi_{y}^{-1}=\lambda$ for all $y\in Y$.
\end{lemma}

\begin{proof}
  Recall that any two Polish spaces of uncountable cardinality are homeomorphic as Borel spaces; cf.\ \cite[Theorem~2.12, p.\,14]{Parthasarathy.67}. As atomless measures can exist only on uncountable spaces, we deduce that there is a Borel homeomorphism $\Phi: (Z,\cB(Z))\to ([0,1],\cB([0,1]))$. Consider $Q_{y}=P_{y}\circ \Phi^{-1}$; then $(Q_{y})$ are atomless probability measures; i.e., their c.d.f.'s $F_{y}(x):=Q_{y}((-\infty,x])$ are continuous in $x$. By construction, they are also measurable in $y$. In particular, as Caratheodory functions, they are jointly measurable in $(x,y)$; cf.\ \cite[Lemma~4.51, p.\,153]{AliprantisBorder.06}. Finally, recall that if $F$ is the c.d.f.\ of a random variable $X$ with continuous distribution, then $F(X)\sim\lambda$. As a result, $\phi_{y}=F_{y}\circ\Phi$ satisfies the requirement of the lemma.
\end{proof}

Combining the two preceding lemmas, we obtain the following.

\begin{corollary}\label{co:borelTransform}
There exists a jointly Borel measurable map 
\[
  C_{0}(\R_{+})\times (-\sqrt{\eta},\sqrt{\eta})\to [0,1],\quad (\omega,h)\mapsto \phi_h(\omega)
\]
such that $\W_{h}\circ \phi_h^{-1}=\lambda$ for each $h$.
\end{corollary}

We can now provide the proof of the proposition.

\begin{proof}[Proof of Proposition \ref{pr:RandomGenerator}]
Let $\xi=\W(d\omega)\xi_\omega(ds)$ be a disintegration of the given randomized stopping time $\xi\in \cR^\eta(\nu)$ and define $\bar \xi=\W(d\omega)\bar\xi_\omega(ds)$ through
\[
  \bar\xi_{f \oplus \omega} := \int_{C_{h(f)}} \xi_{g \oplus \omega} \W_{h(f)}(dg),\quad f, \omega \in C_0(\R_+).\]
Clearly $\bar\xi_{f \oplus \omega}$ depends on $f$ only through $h(f)$. We show below that 
$\bar \xi\in\cR^{\eta}(\nu)$ and $\bar \xi (G)= \xi(G)$ for all $G\in C_b^\eta$. Admitting this for the moment, it remains to construct a non-randomized stopping time with the same law as~$\bar \xi$.
Following Lemma~\ref{le:rho}, we can associate a stopping time $(\omega,u)\mapsto\rho(\omega,u)$ on the probability space $(C_0(\R_+)\times [0,1], \W\otimes \lambda)$ with $\bar\xi$, and we can choose a version of $\rho$ such that
$\rho(f \oplus \omega,u)$ depends on $f$ only through $h(f)$. Finally, let $\phi_{h}$ be as in Lemma~\ref{co:borelTransform}, then by construction, $\tau(\omega):= \rho(\omega, \phi_{h(\omega)}(\omega))$ is a stopping time on $C_0(\R_+)$ such that $\W\circ (B,\tau)^{-1}=\bar\W\circ (\bar B,\rho)^{-1}=\bar\xi$. Thus, its embedding $\xi^{\tau}\in\cR_{\cT}$ is the required non-randomized stopping time.

It remains to verify that $\bar \xi\in\cR^{\eta}(\nu)$ and $\bar \xi (G)= \xi(G)$ for all $G\in C_b^\eta$. 
Let $f\in C_{h}$ and $\omega\in C_0(\R_+)$; then $\bar\xi_{f \oplus \omega}$ is
adapted and concentrated on $[\tau_{\eta}(f),\infty)$) since $\xi_{g \oplus \omega}$ has these properties for all $g\in C_{h}$ (recall that $\tau_{\eta}$ is constant on $C_{h}$). This yields that $\bar\xi\in\cR^{\eta}$.
Next, let $G\in C_{b}^{\eta}$. Then $\bar\xi(G)$ is equal to
\begin{align*}
&\int_{C_0(\R_+)} \int_{\R_+} G(\omega,s) \bar\xi_\omega(ds) \W(d\omega) \\
&= \int_{\R}\int_{C_0(\R_+)}\int_{C_h}\int_{\R_+} G(f \oplus \omega,s) \bar\xi_{f \oplus \omega}(ds) \W_h(df)\W(d\omega)\mu(dh) \\
&= \int_{\R}\int_{C_0(\R_+)}\int_{C_h}\int_{C_h} \int_{\R_+}  G(f \oplus \omega,s) \xi_{g \oplus \omega}(ds) \W_h(dg)\W_h(df)\W(d\omega)\mu(dh) \\
&= \int_{\R}\int_{C_0(\R_+)}\int_{C_h}\int_{C_h} \int_{\R_+}  G(g \oplus \omega,s) \xi_{g \oplus \omega}(ds) \W_h(dg)\W_h(df)\W(d\omega)\mu(dh) \\
&= \int_{C_0(\R_+)} \int_{\R_+} G(\omega,s) \xi_\omega(ds) \W(d\omega) = \xi(G).
\end{align*}
This also implies that $\bar\xi(\bt)=\xi(\bt)<\infty$, since $\bt\wedge n\in C_{b}^{\eta}$ for all $n\in\N$. Finally, to see that $\xi$ and $\bar\xi$ embed the same distribution~$\nu$, we show that $\bar\xi(\phi(B)) = \xi(\phi(B))$ for $\phi\in C_{b}(\R)$. Indeed, consider $G:=\phi(B_{\bt\vee\tau_{\eta}})\in C_{b}^{\eta}$ (recall that adaptedness was not required). Since we already know that $\xi,\bar\xi\in\cR^{\eta}$, we have $\xi(\phi(B))=\xi(G)=\bar\xi(G)=\bar\xi(\phi(B))$ and the proof is complete.\end{proof}

\section{The Dual Problem}\label{se:dualProblem}

We first introduce the domain of the dual problem. 
To that end, recall that~$\nu$ is a centered distribution on $\R$ with finite second moment and let $J\subseteq \R$ be the smallest convex set with $\nu(J)=1$. Thus, a boundary point of~$J$ is contained in~$J$ if and only if it is an atom of $\nu$. We fix%
\footnote{The results below do not depend on the choice of $K_{n}$. When $\nu$ has bounded support with atoms at both endpoints, we can simply take $K_{n}=J$ for all $n$.}
an increasing sequence $(K_{n})$ of compact intervals $0\in K_{n}\subseteq J$ whose union is $J$ and let 
$$
  T_{n}=\inf \{t\geq0: B_{t}\in \partial K_{n}\}
$$
be the first hitting time of the boundary $\partial K_{n}$.
Recall that probabilistic notions are understood with respect to the canonical space $C_{0}(\R_{+})$ and that~$S$ can be embedded in $C_{0}(\R_{+})\times\R$. We fix a (not necessarily measurable) function $G: S\to[0,\infty]$ and introduce the following.

\begin{definition}\label{de:dualDomain}
  Let $\cD(G)$ be the set of all pairs $(M,\psi)$ where the Borel function $\psi: J\to \R \cup\{\infty\}$ is in $L^{1}(\nu)$ and $M$ is a (continuous) local $(\W,\F)$-martingale with $M_{0}=0$ such that 
$$
  M + \psi (B) \geq G \quad\mbox{ on }\quad \cup_{n}[0,T_{n}]\quad\mbox{ (up to evanescence)}
$$
and $M_{\cdot \wedge T_{n}}$ is bounded below for all $n$. The \emph{dual problem} is
$$
  \bI(G)=\inf_{(M,\psi)\in \cD(G)} \nu(\psi).
$$
\end{definition}

The continuity of $M$ refers to its paths being a.s.\ continuous. We recall that Brownian local martingales always admit continuous versions and assume implicitly that all local martingales are continuous in what follows. We also note that $\cup_{n}[0,T_{n}]=\R_{+}\times\Omega$ if $J=\R$. In Section~\ref{se:counterex} it will be shown that the value of the dual problem $\bI(G)$ can change if $M$ is restricted to true martingales or if $\psi$ is restricted to continuous functions. As the relaxations in Definition~\ref{de:dualDomain} are novel, we detail some technical observations.

\begin{remark}\label{rk:horizon}
   Setting  $T=\lim T_{n}=\inf \{t\geq0: B_{t}\in \partial J\}$, we have $\cup_{n}[0,T_{n}]=[0,T]$ if $J$ is closed and $\cup_{n}[0,T_{n}]=[0,T)$ if $J$ is open, but in general either of the inclusions $[0,T)\subseteq \cup_{n}[0,T_{n}]\subseteq [0,T]$ can be strict and the value of $\bI(G)$ can change if $\cup_{n}[0,T_{n}]$ is replaced by $[0,T)$ in Definition~\ref{de:dualDomain}. One can, however, replace $\cup_{n}[0,T_{n}]$ by $[0,T]$ without altering the value of $\bI(G)$, since given $(M,\psi)\in\cD(G)$ we may extend $\psi$ by setting $\psi=\infty$ on $\partial J\setminus J$ without affecting $\nu(\psi)$. (We have nevertheless found it less confusing to write $\cup_{n}[0,T_{n}]$ everywhere.)
\end{remark}

\begin{remark}\label{rk:testingOptional}
  Suppose that $G$ is Borel measurable; i.e., $G$ can be seen as an optional process. Then the inequality $M + \psi (B) \geq G$ on $\cup_{n}[0,T_{n}]$ up to evanescence is equivalent to the almost-sure inequality $M_{\tau} + \psi(B_{\tau}) \geq G_{\tau}$ for all stopping times $\tau$ with $\tau\leq T_{n}$ for some $n$. This follows from the optional cross-section theorem; cf.\ \cite[Theorem~IV.84, p.\,137]{DellacherieMeyer.78}.
\end{remark}

\begin{remark}\label{rk:sampling}
  Using Fatou's lemma, boundedness from below of $M_{\cdot \wedge T_{n}}$ implies that for all stopping times $\sigma\leq\tau\leq T_{n}$, the random variables $M_{\sigma},M_{\tau}$ are integrable and satisfy the optional sampling property $E[M_{\tau}|\cF_{\sigma}]\leq M_{\sigma}$.
\end{remark}

\begin{lemma}\label{le:superhedUptoEmbedding}
Let $(M,\psi)\in \cD(G)$. (i) We have $\W\{\tau\leq T_{n}\}\to 1$ and $M + \psi (B) \geq G$ on $[0,\tau]$ for all $\tau\in\cT(\nu)$. (ii) We have $M + \psi (B) \geq G$ $\xi$-a.s.\ for all $\xi\in\cR(\nu)$.
\end{lemma}

\begin{proof}
  (i) Let $\tau\in\cT(\nu)$; in particular, $E[\tau]<\infty$. Since $B_{\cdot \wedge \tau}$ is a uniformly integrable martingale and $B_{\tau}$ is supported on $J$, we have $\W\{\tau\leq T_{n}\}\to 1$ and hence $[0,\tau]\subseteq \cup_{n}[0,T_{n}]$. In particular, $M + \psi (B) \geq G$ on $[0,\tau]$ up to evanescence. 
  
  (ii) Consider the stopping time $\rho$ associated with $\xi$ on the extended space~$\bar{C_{0}}(\R_{+})$; cf.\ Lemma~\ref{le:rho}. As $\bar\W\circ (\bar B,\rho)^{-1}=\xi$, we can deduce the claim by applying the arguments for (i) to $\rho$ and $\bar B$.
\end{proof}

\begin{lemma}\label{le:envelope}
  Let $(M,\psi)\in\cD(0)$ and write $\psi^{**}$ for the convex envelope on~$J$. Then $(M,\psi^{**})\in\cD(0)$; i.e.,
  \begin{equation}\label{eq:ConvexHullSuperhed}
    M + \psi^{**}(B)\geq 0 \quad\mbox{ on }\quad \cup_{n}[0,T_{n}],
  \end{equation}
  and in particular $\psi^{**}(0)\geq0$.
\end{lemma}

\begin{proof}
  We first observe that for $x\in J$,
  $$
    \psi^{**}(x) = \inf \big\{\tfrac{b}{a+b} \psi(x-a) + \tfrac{a}{a+b} \psi(x+b):\, a,b\in\R_{+},\; [x-a,x+b]\subseteq J\big\}
  $$
  where the sum is understood as $\psi(x)$ if $a=b=0$. Suppose that~\eqref{eq:ConvexHullSuperhed} fails, then by the optional cross-section theorem there exist $n\geq1$ and a stopping time $\sigma\leq T_{n}$ such that
  $$
     \W\{M_{\sigma} + \psi^{**}(B_{\sigma})<0\}>0. 
  $$
  By the above formula for $\psi^{**}(x)$ and a measurable selection argument, it follows that there exist $m\geq n$ and $\cF_{\sigma}$-measurable random variables $a,b\geq0$ with $[B_{\sigma}-a,B_{\sigma}+b]\subseteq K_{m}$ such that 
  \begin{equation}\label{eq:proofEnvelopeContrad}
    \W\big\{ M_{\sigma} + \tfrac{b}{a+b} \psi(B_{\sigma}-a) + \tfrac{a}{a+b} \psi(B_{\sigma}+b) < 0\big\} >0.
  \end{equation}
  Let $\tau=\inf\big\{t\geq \sigma:\, B_{t}-B_{\sigma}\notin (-a,b)\big\}$ and note that $\tau\leq T_{m}$. Then optional sampling (cf.\ Remark~\ref{rk:sampling}) implies
  $$
    M_{\sigma} + \tfrac{b}{a+b} \psi(B_{\sigma}-a) + \tfrac{a}{a+b} \psi(B_{\sigma}+b) \geq E[M_{\tau}+\psi(B_{\tau})|\cF_{\sigma}].
  $$
  But $(M,\psi)\in\cD(0)$  implies that $M_{\tau}+\psi(B_{\tau})\geq0$, and now a contradiction to~\eqref{eq:proofEnvelopeContrad} ensues. In particular, $\psi^{**}(0)\geq0$. Hence, the convex function $\psi^{**}$ is bounded from below by a linear function and from above by $\psi$, so that $\psi^{**}\in L^{1}(\nu)$.
\end{proof}

The following normalization will be used repeatedly below.

\begin{remark}\label{rk:shift}
  Let $(M,\psi)\in\cD(G)$ and $c\in\R$. Define $M'=M+cB$ and $\psi'(x)=\psi(x)-cx$. Then $(M',\psi')\in\cD(G)$ since $M'+\psi'(B)=M+\psi(B)$ and $B$ is a bounded martingale on $[0,T_{n}]$ with $B_{0}=0$.
  
  Using this with the (left, say) derivative $c:=\partial^{-} \psi^{**}(0)$ and recalling that $\psi^{**}(0)\geq0$ by Lemma~\ref{le:envelope}, we see that any $(M,\psi)\in\cD(G)$ can be normalized such that $\psi^{**}\geq0$ and hence also $\psi\geq0$.
\end{remark}

Next, we establish the key inequality for the ``weak'' duality and in particular that our definition of the dual domain is rigid enough despite the relaxations.

\begin{lemma}\label{le:Msupermart}
  Let $(M,\psi)\in\cD(0)$. (i) We have $E[M_{\tau}]\leq0$ for all $\tau\in\cT(\nu)$.
  (ii) We have $\xi(M)\leq0$ for all $\xi\in\cR(\nu)$.
\end{lemma}

\begin{proof}
  (i) We prove the claim for any stopping time $\tau$ such that $B_{\cdot\wedge\tau}$ is uniformly integrable,  $E[\psi(B_{\tau})]<\infty$ and $[0,\tau]\subseteq \cup_{n}[0,T_{n}]$; in particular, these are satisfied for $\tau\in\cT(\nu)$. Indeed, by Remark~\ref{rk:shift}, we can assume without loss of generality that $\psi\geq0$ and then $0\leq\psi^{**}\leq \psi$. Since $\psi^{**}$ is convex and $\psi^{**}(B_{\tau})$ is integrable, it follows that  $\psi^{**}(B_{\cdot\wedge \tau})$ is a nonnegative uniformly integrable submartingale, hence of class~(D). On the other hand, Lemma~\ref{le:envelope} yields that
  $M \geq -\psi^{**}(B)$ on $[0,\tau]$, so we conclude that $M^{-}_{\cdot\wedge\tau}$ is of class~(D). Fatou's lemma (in its version with a uniformly integrable lower bound) now implies that $E[M_{\tau}]\leq0$.

  (ii) Let $\xi\in\cR(\nu)$ and let $\rho$ be the stopping time associated with $\xi$ on the extended space $\bar{C_{0}}(\R_{+})$; cf.\ Lemma~\ref{le:rho}. Then $\xi\in\cR(\nu)$ implies that $\rho$ is an embedding of $\nu$ into $\bar B$ and the argument of~(i) applied on $\bar{C_{0}}(\R_{+})$ yields $\xi(M)=E^{\bar\W}[M(\bar{B})_{\rho}]\leq0$.
\end{proof}

Definition~\ref{de:dualDomain} is tailored to imply the following closedness property; it will be used to derive the absence of a duality gap in the next section and immediately implies the existence of a dual optimizer.

\begin{proposition}\label{pr:closedness}
  Consider $G^{k},G:S\to[0,\infty]$ and $(M^{k},\psi_{k})\in \cD(G^{k})$ for $k\geq1$. Assume that $G^{k}\to G$ pointwise and $\sup_{ k}\nu(\psi_{k})<\infty$. Then there exist $(M,\psi)\in \cD(G)$ such that $\nu(\psi)\leq \liminf \nu(\psi_{k})$.
\end{proposition}

\begin{proof}
  Let $c=\sup \nu(\psi_{k})$. By passing to a subsequence we may assume that $\lim \nu(\psi_{k})$ exists.
  By Remark~\ref{rk:shift}, we may normalize $\psi_{k}$ such that $\psi^{**}_{k}\geq0$. 
  
  As $\psi_{k}\geq \psi^{**}_{k}\geq0$, Komlos' lemma  (see \cite[Lemma A1.1]{DelbaenSchachermayer.94} and its subsequent remark)  allows us to find convex combinations $\phi_{k}$ of $(\psi_{k},\psi_{k+1},\dots)$ which converge $\nu$-a.s. Let $\psi:=\limsup \phi_{n}$ and note that 
  $$
    0\leq \nu(\psi)=\nu(\liminf \phi_{n})\leq \lim \nu(\phi_{k})= \lim \nu(\psi_{k})
  $$
  by Fatou's lemma. After replacing $M_{k}$ with the corresponding convex combinations, we may assume that $ \phi_{k}=\psi_{k}$.
  
  We have $0\leq \nu(\psi^{**}_{k}) \leq \nu(\psi_{k}) \leq c$.
  As in the proof of \cite[Proposition~5.5]{BeiglbockNutzTouzi.15}, this implies a uniform (in $k$) bound for the Lipschitz constant of $\psi^{**}_{k}$ on each compact subset of $J$, and hence a uniform bound $0\leq \psi^{**}_{k}\leq c_{n}$ on $K_{n}$, independent of $k$. By Lemma~\ref{le:envelope}, it follows that
  $$
    M^{k} \geq -\psi^{**}_{k}(B)\geq -c_{n} \quad\mbox{on}\quad  [0,T_{n}].
  $$
  This guarantees that $(M^{k})$ admit a limiting supermartingale $Z$ in a suitable sense. More precisely, \cite[Theorem~2.7]{CzichowskySchachermayer.16} and a diagonal argument show that after taking suitable convex combinations $N^{k}$ of $(M^{k},M^{k+1},\dots)$, there exists an optional%
\footnote{
To be completely precise, \cite{CzichowskySchachermayer.16} assumes the usual conditions for the filtration. We can, for instance, apply their result in the $\W$-augmented filtration $\F^{a}$ and then pass back to $\F$ at the end of this paragraph: the local martingale $M$ is necessarily continuous and hence $\F^{a}$-predictable, but then we can choose an $\F$-predictable version (up to evanescence) of $M$ by applying~\cite[Appendix~I.7, p.\,399]{DellacherieMeyer.82}.
}
process $Z$ which is a strong supermartingale (as defined in \cite[Appendix~I]{DellacherieMeyer.82}) on $[0,T_{n}]$ for all $n$ and
  $$
    N^{k}_{\tau}\to Z_{\tau}  \quad\mbox{in probability}
  $$ 
  for any stopping time $\tau$ such that $\tau\leq T_{n}$ for some $n$. We may assume that $N^{k}=M^{k}$. As $Z_{0}=0$, the process $Z$ admits a Mertens decomposition $Z=M-A$ where $A$ is nondecreasing, $M$ is a local martingale and $A_{0}=M_{0}=0$; cf.\ \cite[Appendix~I.20, p.\,414]{DellacherieMeyer.82}.
  
  It remains to show that $(M,\psi)\in \cD(G)$. Set $H=\limsup_{k} (M^{k}+\psi_{k}(B))$; then $H$ is optional and $H\geq G$ on $\cup_{n}[0,T_{n}]$ up to evanescence. Thus, it suffices to show that $(M,\psi)\in \cD(H)$. Indeed, we have 
  $$
    M_{\tau}\geq Z_{\tau} \geq -c_{n} \quad\mbox{and}\quad M_{\tau} + \psi (B_{\tau}) \geq Z_{\tau} + \psi (B_{\tau}) \geq H_{\tau}
  $$
  for all $\tau\leq T_{n}$. As $H$ is optional, we conclude that $(M,\psi)\in \cD(H)$ by using Remark~\ref{rk:testingOptional}.
\end{proof}

Applying Proposition~\ref{pr:closedness} to a constant sequence, we deduce that dual existence holds for general reward functions, in contrast to previous formulations of the dual problem~\cite{BeiglbockCoxHuesmann.14,GuoTanTouzi.15a} where existence can fail even for more regular reward functions.

\begin{corollary}\label{co:dualExistence}
  Let $G: S\to [0,\infty]$. If $\bI(G)<\infty$, there exists a dual optimizer $(M,\psi)\in \cD(G)$.
\end{corollary}

\section{Duality}\label{se:duality}

In this section we combine the closedness result of Proposition~\ref{pr:closedness} with capacity theory and facts about optimal Skorokhod embeddings to establish the absence of a duality gap. We first state the weak duality.
  
\begin{lemma}\label{le:weakDuality}
  Let $G: S\to [0,\infty]$. Then $\bS(G)\leq \bI(G)$.
\end{lemma}
\begin{proof}
  We need to show that $\nu(\psi)\geq \xi(G)$ whenever $(M,\psi)\in\cD(G)$ and $\xi\in\cR(\nu)$. Indeed, by Lemma~\ref{le:superhedUptoEmbedding} we have $M +\psi(B)\geq G$ $\xi$-a.s. Taking expectations under $\xi$ and recalling that $\xi(M)\leq0$ by Lemma~\ref{le:Msupermart}, the claim follows by the monotonicity of the (outer) integral.
\end{proof}

Next, we state a duality result for semicontinuous functions. The following is a consequence of
\cite[Theorem~1.2]{BeiglbockCoxHuesmann.14} as well as of \cite[Theorem~2.4]{GuoTanTouzi.15a} after noting that their dual domain is a subset of ours. We provide a sketch of proof for the convenience of the reader.

\begin{proposition}\label{pr:strongDualityCont}
Let $G: S\to\R$ be bounded and upper semicontinuous. Then $\bS(G) = \bI(G)$.
\end{proposition}

\begin{proof}[Sketch of Proof.]
  The inequality $\bS(G)\leq \bI(G)$ holds by Lemma~\ref{le:weakDuality}. We sketch an argument for the reverse inequality following~\cite{GuoTanTouzi.15a} (and refer to the latter for details). 
The key idea is to dualize the constraint~$\nu$ and use the Fenchel--Moreau theorem. Indeed, let $\cV$ be the set of all centered probability measures $\nu$ with finite first moment, equipped with the 1-Wasserstein metric. One verifies that if $\nu_{n}\to\nu$ in $\cV$ and $\xi_{n}\in\cR(\nu_{n})$, then there exists $\xi\in\cR(\nu)$ which is a weak limit of a subsequence $(\xi_{n_{k}})$. Let $\nu\in\cV$ and let $\bS(\nu)=\bS(G,\nu)$ be the corresponding primal problem. Taking $\xi_{n}\in\cR(\nu_{n})$ to be a $(1/n)$-optimizer for the primal problem $\bS(\nu_{n})$; i.e.,
$$
  \xi_{n}(G) \geq \sup_{\xi'\in\cR(\nu_{n})} \xi'(G) -1/n,
$$
it follows that
$
  \bS(\nu)\geq \xi(G) \geq \limsup \bS(\nu_{n}).
$
In brief, $\nu\mapsto \bS(\nu)$ is upper semicontinuous on $\cV$, and clearly it is also concave and finite-valued. The space $\cV$ can be seen as a closed convex subspace of a Hausdorff locally convex vector space~$V$ of signed measures and $\nu\mapsto \bS(\nu)$ can be extended to~$V$ by assigning the value $-\infty$ outside of $\cV$. The topological dual is $V^{*}=\cC^{1}$, the space of continuous functions $\psi:\R\to\R$ with linear growth. The Fenchel--Moreau theorem then shows that $\bS(\nu)$ is equal to its biconjugate,
\begin{equation}\label{eq:Fenchel}
  \bS(\nu)=\bS^{**}(\nu)=\inf_{\psi\in\cC^{1}} \left[\sup_{\nu'\in\cV} \bS(\nu') - \nu'(\psi)\right] + \nu(\psi).
\end{equation}
Let us fix $\psi\in\cC^{1}$ and focus on the inner optimization,
$$
  \sup_{\nu'\in\cV} \bS(\nu') - \nu'(\psi) = \sup_{\nu'\in\cV} \sup_{\xi\in\cR(\nu')} \xi(G-\psi(B)).
$$
Seeing the functional $Y:=G-\psi(B)$ as the reward function of an optimal stopping problem, one can check that
$$
  v_{0}(\psi):=\sup_{\nu'\in\cV} \sup_{\xi\in\cR(\nu')} \xi(Y) = \sup_{\xi\in\cR} \xi(Y)=\sup_{\tau\in\cT} E[Y_{\tau}]
$$
is simply the value function of the associated standard optimal stopping problem. In particular, $v_{0}=v_{0}(\psi)$ is the initial value of the associated Snell envelope $Z$; i.e., the minimal supermartingale dominating~$Y$. We write its Doob--Meyer decomposition as 
$
  Z=v_{0}+M-A
$
where $M$ is a local martingale, $A$ is increasing and $A_{0}=M_{0}=0$.
Then 
$$
  v_{0}+M\geq Y \quad\mbox{or equivalently}\quad v_{0}+M + \psi(B)\geq G.
$$
Since $\psi$ has linear growth and $G$ is bounded, $v_{0}+M + \psi(B)\geq G$ implies that $M_{\cdot \wedge T_{n}}$ is bounded below for all $n$. As a result, $(M,\bar\psi)\in\cD(G)$ for the function $\bar\psi:=v_{0}+\psi$. Since this holds for all $\psi\in\cC^{1}$, \eqref{eq:Fenchel} yields that
$$
  \bS(\nu)=\bS^{**}(\nu)=\inf_{\psi\in\cC^{1}} v_{0}(\psi) + \nu(\psi)  \geq \inf_{(\bar\psi,M)\in\cD(G)} \nu(\bar\psi)=\bI(G)
$$
as desired.
\end{proof}

On the strength of the closedness property in Proposition~\ref{pr:closedness}, we can use capacity theory to extend the duality to measurable functions.
Let $[0,\infty]^{S}$ be the set of all functions $G: S\to [0,\infty]$, let $\USA_{+}$ be the sublattice of upper semianalytic%
\footnote{The function $G$ is called upper semianalytic if the sets $\{G\geq c\}$ are analytic for all $c\in\R$, where a subset of $S$ is called analytic if it is the image of a Borel subset of a Polish space under a Borel mapping. Any Borel function is upper semianalytic and any upper semianalytic function is universally measurable. See, e.g., \cite[Chapter~7]{BertsekasShreve.78} for background.
}
functions and let $\cU$ be the sublattice of bounded upper semicontinuous functions; note that $\cU$ is stable with respect to countable infima. A mapping $\bC: [0,\infty]^{S} \to [0,\infty]$ is called a $\cU$-capacity if it is monotone, sequentially continuous upwards on $[0,\infty]^{S}$ and sequentially continuous downwards on~$\cU$.

\begin{lemma}\label{le:SisCapacity}
  The mapping $\bS: [0,\infty]^{S}\to [0,\infty]$ is a $\cU$-capacity.
\end{lemma}

\begin{proof}
  As $\xi(\bt)=\int x^{2}\,d\nu$ for all $\xi\in\cR(\nu)$, the set $\cR(\nu)$ is a nonempty compact space of probability measures on~$S$; see for instance \cite[Theorem~3.14]{BeiglbockCoxHuesmann.14}. This implies that the associated sublinear map $G\mapsto \sup_{\xi\in\cR(\nu)} \xi(G)\equiv \bS(G)$ is a capacity. Indeed, continuity upwards is immediate by commuting two suprema. Let $G_{n}\in\cU$ decrease to $G\in\cU$, then there are $\xi_{n}\in\cR(\nu)$ such that $\xi_{n}(G_{n})\geq \bS(G_{n})-1/n$. After passing to a subsequence, $\xi_{n}\to\xi$ weakly for some $\xi\in\cR(\nu)$. Then $\bS(G)=\lim_{m}\xi(G_{m})$ and for each~$m$ we have $\xi(G_{m})\geq \limsup_{n} \xi_{n}(G_{m})\geq \bS(G_{m})$; thus $\bS(G)\geq\limsup_{m}\bS(G_{m})$. The reverse holds by monotonicity.
\end{proof}

\begin{lemma}\label{le:IisCapacity}
  The mapping $\bI: [0,\infty]^{S} \to [0,\infty]$ is a $\cU$-capacity.
\end{lemma}

\begin{proof}
  As $\bI=\bS$ on $\cU$ by Proposition~\ref{pr:strongDualityCont}, Lemma~\ref{le:SisCapacity} already shows that~$\bI$ is sequentially continuous downwards on $\cU$. Let $G,G_{n}\in [0,\infty]^{S}$ be such that $G_{n}$ increases to $G$ pointwise; we need to show that $\bI(G_{n})\to \bI(G)$. It is clear that $\bI$ is monotone; in particular, $\bI(G)\geq \limsup \bI(G_{n})$, and $\bI(G_{n})\to \bI(G)$ if $\sup_{n} \bI(G_{n})=\infty$.
  
  Hence, we only need to show $\bI(G) \leq \liminf \bI(G_{n})$ under the condition that $\sup_{n} \bI(G_{n})<\infty$. Indeed, by the definition of $\bI(G_{n})$ there exist $(M^{n},\psi_{n})\in \cD(G_{n})$ with 
  $
    \nu(\psi_{n}) \leq \bI(G_{n}) +1/n.
  $
  Proposition~\ref{pr:closedness} then yields $(M,\psi)\in \cD(G)$ with 
  $
    \nu(\psi) \leq \liminf [\bI(G_{n}) +1/n],
  $
  showing that $\bI(G) \leq \liminf \bI(G_{n})$.
\end{proof}

We can now prove the main duality result. (Recall that dual attainment was already stated in Corollary~\ref{co:dualExistence}, without measurability assumptions.)

\begin{theorem}\label{th:duality}
  Let $G: S\to [0,\infty]$ be upper semianalytic. Then there is no duality gap: $\bS(G)= \bI(G) \in [0,\infty]$.
\end{theorem} 

\begin{proof}
 In view of Lemma~\ref{le:SisCapacity}, %
  Choquet's capacitability theorem (see for instance \cite[Proposition~2.11]{Kellerer.84}) shows that
  $$
    \bS(G)=\sup \{\bS(G'):\, G'\in \cU, \,G'\leq G\},\quad G\in\USA_{+}.
  $$
  By Lemma~\ref{le:IisCapacity}, the analogue holds for $\bI$, and hence the fact that $\bS=\bI$ on~$\cU$ by Proposition~\ref{pr:strongDualityCont} already implies that $\bS=\bI$ on $\USA_{+}$.
\end{proof}

We deduce the following characterization of primal and dual optimizers.

\begin{corollary}\label{co:optimizerEquality}
  Let $G: S\to [0,\infty]$ be upper semianalytic and $\bS(G)<\infty$. Given $(M,\psi)\in\cD(G)$ and $\xi\in\cR(\nu)$, the following are equivalent:
  \begin{enumerate}
  \item $(M,\psi)$ is optimal for $\bI(G)$ and $\xi$ is optimal for $\bS(G)$,
  \item $M + \psi(B) = G$ $\xi$-a.s.\ and $\xi(M)=0$,
  \item $M + \psi(B) = G$ $\xi$-a.s.\ and $\xi(M)\geq0$.
  \end{enumerate}
\end{corollary}

\begin{proof}
  Theorem~\ref{th:duality} shows that
  \[
    \xi(\psi(B))=\nu(\psi)\geq \bI(G)=\bS(G)\geq \xi(G).
  \]
  Given~(iii), we have $\xi(\psi(B))\leq\xi(G)$ and thus the above must be equalities; that is, (i) holds.  Given~(i), these are again all equalities and then~(ii) follows after recalling that $M +\psi(B)\geq G$ $\xi$-a.s.\ and $\xi(M)\leq 0$; cf.\ Lemma~\ref{le:superhedUptoEmbedding} and Lemma~\ref{le:Msupermart}. The  implication from (ii) to (iii) is trivial.
\end{proof}

\begin{remark}\label{rk:lowerBound}
  The lower bound on $G$ in our main duality results can be relaxed to the following condition: there exist $\psi\in L^{1}(\nu)$ and a local martingale $M$ which is bounded on $[0,T_{n}]$ for all $n$ such that
  $$
    G \geq -M - \psi(B).
  $$
  Indeed, the stated results can then be applied to $G':= G+ M + \psi(B)\geq0$ to derive the corresponding assertions for~$G$.
\end{remark}

Next, we provide a monotonicity principle in the spirit of~\cite[Corollary~7.8]{BeiglbockNutzTouzi.15} which provides a universal support~$\Gamma$ characterizing all optimal embeddings, under an integrability condition on~$G$. As mentioned in the Introduction, this complements the monotonicity principle of~\cite{BeiglbockCoxHuesmann.14} which gives a geometric condition on the support that is necessary for optimality, but not sufficient. The following condition is necessary and sufficient. However, the geometry of~$\Gamma$ is merely described in a weaker form, through the construction via a suitable dual optimizer in~\eqref{eq:GammaDefn}. We will exemplify in Section~\ref{se:cave} how to exploit such a description. 
For the statement, note that while $\xi\in\cR$ is defined as a measure on $C_{0}(\R_{+})\times \R_{+}$, it naturally induces a measure on~$S$: for $\Gamma\in\cB(S)$, we set $\xi(\Gamma)=\xi\{(\omega,t)\in C_{0}(\R_{+})\times \R_{+}:\,  (\omega|_{[0,t]},t)\in\Gamma\}$.

\begin{corollary}\label{co:monotPrinciple}
  Let $G:S \to [0,\infty]$ be Borel and of class~(D). There exists a Borel set $\Gamma \subseteq S$ such that a randomized stopping time $\xi \in \cR(\nu)$ is optimal for $\bS(G)$ if and only if it is concentrated on $\Gamma$; i.e., $\xi(\Gamma)=1$.
\end{corollary}

\begin{proof}
  Since $G$ is of class (D), there exists a class~(D) martingale $N$ such that $G\leq N$; cf.\ \cite[Appendix~I.24, p.\,419]{DellacherieMeyer.82}. In particular, $\xi(G)\leq \xi(N)=N_{0}$ for all $\xi\in\cR(\nu)$, showing that $\bS(G)<\infty$.
  Let $(M',\psi) \in \cD(G)$ be a dual optimizer as guaranteed by Corollary~\ref{co:dualExistence}, normalized
  such that $\psi \geq 0$. Define
  \[ 
    M = M'\1_{[0,\tau]} + N\1_{[\tau,\infty)} \quad \mbox{where}\quad  \tau = \inf\{t \geq 0:\,M'_t = N_{t}\}.
  \]
  Then $(M,\psi) \in \cD(G)$ is again a dual optimizer and as $M^{+}$ is of class~(D), Fatou's lemma (in its version with a uniformly integrable bound) yields that $\xi(M)\geq0$ for any $\xi\in\cR$.
  We set
  \begin{equation}\label{eq:GammaDefn}
    \Gamma:=\{M + \psi(B) = G\}\subseteq S,
  \end{equation}
  then the equivalence of (i) and (iii) in Corollary~\ref{co:optimizerEquality} shows that $\xi\in\cR(\nu)$ is optimal if and only if $\xi(\Gamma)=1$.
\end{proof}

In the above proof, the integrability of $G$ is used to infer that there exists a dual optimizer $(M,\psi)$ such that $\xi(M)=0$ for all $\xi\in\cR(\nu)$. We will see in Section~\ref{se:monPrinciple} that this need not hold when $G$ is not of class~(D), even if $\bS(G)<\infty$. More surprisingly, Corollary~\ref{co:monotPrinciple} and even the very essence of the monotonicity principle may fail: the optimality of an optimal embedding \emph{cannot be described through its support} in that case.

\section{Optimal Cave Embeddings}\label{se:cave}

In this section, we consider an optimal Skorokhod embedding problem introduced in~\cite{BeiglbockCoxHuesmann.14} where the (unique) optimal embedding is the hitting time  of a set that is cave-shaped; that is, consists of a left and a right barrier. There are typically infinitely many such caves that embed a given distribution $\nu$, leading to the question of how to characterize the optimal one. Our main purpose here is to exemplify how our result on dual existence can be useful in deriving a variational characterization. The characterization itself is similar to the one provided by Cox and Kinsley in~\cite{CoxKinsley.18} for a different class of reward functions. However, thanks to dual existence, our argument here is much more direct than the one of Cox and Kinsley who develop a discretization approach~\cite{CoxKinsley.18b} in order to reduce to finite-dimensional linear programming problems and obtain their result through tedious and delicate limiting arguments. More importantly, the proof given here suggests that \emph{the variational condition holds for much more general classes of embeddings,} an issue to be addressed in future work.

We first fix some terminology (see also \cite{BeiglbockCoxHuesmann.14, CoxKinsley.18, Obloj.04}). Let $(\R_{+}\times \R)^{*}=(\R_{+}\times \R)\setminus\{(0,0)\}$ be the punctured half-plane. Following Root's embedding, a \emph{right barrier} is a closed set $\cR\subseteq (\R_{+}\times \R)^{*}$ which is closed to the right; that is, $(s,x)\in\cR$ and $t\geq s$ imply $(t,x)\in\cR$. Such a barrier is characterized by a function $x\mapsto r(x)\in [0,\infty]$ which traces its left boundary,
$$
  r(x) = \inf\{t\geq 0:\, (t,x)\in\cR\},\quad \inf\emptyset :=+\infty.
$$
Similarly, in the spirit of Rost, a left barrier $\cL$ is a closed subset of $(\R_{+}\times \R)^{*}$ which is closed to the left. 
It is characterized by a function $x\mapsto l(x)\in \{-1\}\cup[0,\infty)$ where we now set
$$
  l(x) = \sup\{t\geq 0:\, (t,x)\in\cL\},\quad \sup\emptyset :=-1.
$$
Note that we use the value $-1$ for gaps in the left barrier.

\begin{definition}\label{de:cave}
  Given $t_{p}\in\R_{+}$, a \emph{cave barrier with parting} $t_{p}$ is a set $\cL\cup\cR$ where $\cL\subseteq [0,t_{p}]\times \R$ is a left barrier and $\cR\subseteq [t_{p},\infty)\times \R$ is a right barrier.
\end{definition}

Clearly, a cave barrier $\cL\cup\cR$ is characterized by two functions $l\leq t_{p}\leq r$. We denote by $D$ the (open) complement $(\R_{+}\times \R)^{*}\setminus (\cL\cup\cR)$ and refer to $D$ as the \emph{continuation region} of the cave barrier (or just as the cave, when there is no ambiguity). Let $\tau=\inf\{t\geq 0:\, (t,B_{t})\notin D\}$ be the first exit time of $D$. This is the minimum of the two stopping times
\begin{align*}
  \tau_{l} &= \inf\{t\geq 0:\, B_{t}\in \cL\} \in [0,t_{p}]\cup\{\infty\},\\
  \tau_{r} &= \inf\{t_{p}\leq t <\tau_{l}:\, B_{t}\in \cR\}\in [t_{p},\infty].
\end{align*}
Similarly, if $\nu=\Law(B_{\tau})$ is the measure embedded by the cave barrier, then $\nu=\nu_{l}+\nu_{r}$ is decomposed into the subprobabilities corresponding to mass absorbed at the left and right barrier; or more precisely, $\nu_{l}=\Law(B_{\tau_{l}})$ and $\nu_{r}=\Law(B_{\tau_{r}})$ with the convention that $B_{\infty}$ is valued in an external cemetery state and the laws are restricted to $\R$.

Two different cave barriers may have the same hitting time. First, consider a left barrier $\cL$ with corresponding function $l$. If $l$ has a ``sink'' on $(0,\infty)$, say, then $(t,B_{t})$ cannot hit that part of the boundary since the time coordinate always runs forward (see Figure~\ref{fi:caveRegular}). If we replace $l$ by the increasing envelope of $x\mapsto l(x)$ on $(0,\infty)$ and its decreasing envelope on $(-\infty,0)$, the new barrier has the same hitting time. We call $\cL$ \emph{monotone} if $l$ is already equal to this envelope. Note that given $\cL$, there exists a minimal monotone left barrier containing $\cL$. Next, consider a cave barrier and note that $(t,B_{t})$ can only hit the boundary of the component of $D$ that contains $(0,0)$. Thus, we
say that the cave is \emph{connected} if~$D$ is. For brevity we say that a cave barrier is \emph{regular} if~$\cL$ is monotone and $D$ is connected, and note that every cave barrier has a minimal regular cave barrier containing it.
\begin{figure}
\subfigure[~]{\label{fi:a}
\begin{minipage}[l]{0.482\textwidth}
\hspace{-5pt}\vspace{-3.5pt}\resizebox{\textwidth}{!}{
\begin{tikzpicture}
\draw[domain=-4.2:-1,smooth,variable=\y] plot ({0.2*(-7 + 13*\y*\y + 7*\y*\y*\y + \y*\y*\y*\y)},{\y});
\fill[lightgray,domain=-4.2:-1,smooth,variable=\y] plot ({0.2*(-7 + 13*\y*\y + 7*\y*\y*\y + \y*\y*\y*\y)},{\y}) -- (0,-4.2)--(3,-4.2);

\draw[domain=1.4:2.7,smooth,variable=\y] plot ({0.5*\y*\y-1)},{\y});
\fill[lightgray,domain=1.4:2.7,smooth,variable=\y] plot ({0.5*\y*\y-1)},{\y}) -- (0,2.7)--(0,1.4);
\draw[domain=-4.2:2.7,smooth,variable=\y] plot ({-1/(\y-3.2)+17-0.2*(\y+7)*(\y+7)+2*\y},{\y});
\fill[lightgray,domain=-4.2:2.7,smooth,variable=\y] plot ({-1/(\y-3.2)+17-0.2*(\y+7)*(\y+7)+2*\y},{\y}) -- (10,2.7)--(10,-4.2);

\fill[gray,domain=-3.79:-2,smooth,variable=\y] plot ({0.2*(-6.98 + 13*\y*\y + 7*\y*\y*\y + \y*\y*\y*\y)},{\y}) -- (1,-3.79);

\draw (1,-2) -- (1,-3.8) {};

\draw (0,-4.5) -- (0,3) {};
\draw (3,-4.5) -- (3,3) {};
\draw (-0.2,-4.2) -- (10,-4.2) {};
\draw (-0.2,2.7) -- (10,2.7) {};
\draw (-0.2,0) -- (10,0) {};

\node at (-0.7,-4.2) {$x_{min}$};
\node at (-0.7,2.7) {$x_{max}$};
\node at (-0.7,0) {$0$};
\node at (3.3,-4.5) {$t_{p}$};

\node at (0.5,2.2) {$\cL$};
\node at (0.5,-2) {$\cL$};
\node at (9,-2) {$\cR$};
\end{tikzpicture}
}
\end{minipage}
}
\subfigure[~]{\label{fi:b}
\begin{minipage}[c]{0.475\textwidth}
\resizebox{\textwidth}{!}{
\begin{tikzpicture}
\draw (0,-4.5) -- (0,3) {};
\draw (-0.2,-4.2) -- (10,-4.2) {};
\draw (-0.2,2.7) -- (10,2.7) {};
\draw (-0.2,-0.75) -- (0,-0.75) {};

\draw (3,-4.2) -- (3,-2.475) -- (-0.2,-2.475) {};
\fill[lightgray] (3,-4.2) -- (3,-2.475) -- (0,-2.475) -- (0,-4.2) {};
\draw (2,2.7) -- (3,0.975) -- (-0.2,0.975) {};
\fill[lightgray] (3,2.7) -- (3,0.975) -- (0,0.975) -- (0,2.7) {};

\draw (6,0.975) -- (10,0.975) {};
\draw (6,-2.475) -- (10,-2.475) {};

\node at (-0.7,0.975) {$1$};
\node at (-0.7,2.7) {$2$};
\node at (-0.7,-2.475) {$-1$};
\node at (-0.7,-4.2) {$-2$};
\node at (-0.7,-0.75) {$0$};
\end{tikzpicture}
}
\end{minipage}
}
\caption{(a) The dark area is added to the left barrier $\cL$ to make it monotone. (b) Example of Remark~\ref{rk:caveNonUnique}.}
\label{fi:caveRegular}
\vspace{-1em}%
\end{figure}
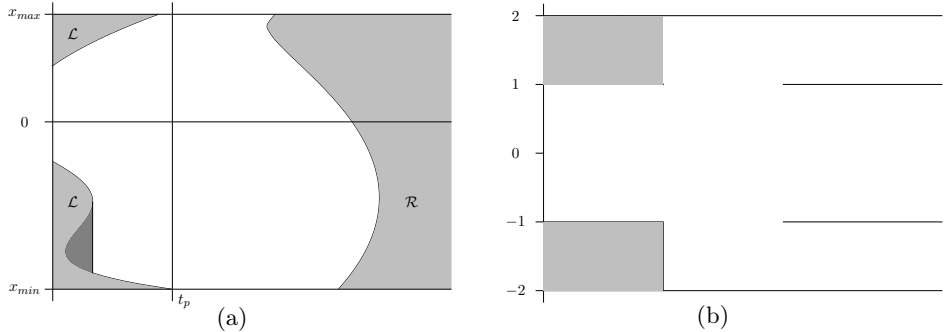
 This notion is important due to the following fact: the first hitting times of two regular caves are a.s.\ equal if and only if the caves are equal. We omit the details and refer instead to \cite[p.\,365]{Obloj.04} and the end of the proof of \cite[Theorem~2.4]{BeiglbockCoxHuesmann.14} for analogous and detailed discussions.

It is also useful to note that an interval of constancy of $l$ corresponds to an interval where $\nu_{l}$ has no mass. On the other hand, $\nu_{r}$ has no mass on an interval where $r=\infty$. Finally, an atom in $\nu$ is generated by a horizontal portion in the boundary of $\cL$ or $\cR$; that is, a discontinuity of $l$ or $r$.

\begin{remark}\label{rk:caveNonUnique}
  Cave embeddings for a given distribution $\nu$ are trivially non-unique because any cave can be regularized without changing the embedded distribution. However, in contrast to Root and Rost embeddings, \emph{cave embeddings are non-unique even if this regularity is imposed}. 
  
  In essence,  this ambiguity arises because a given ``piece'' of $\nu$ can be absorbed either by the left or the right boundary. A simple example can be generated by taking $\nu=(\delta_{-2}+\delta_{-1}+\delta_{1}+\delta_{2})/4$, so that the left and right barriers are (the envelope of) horizontal spikes at these four locations; cf.\ Figure~\ref{fi:caveRegular}. One can shorten the spikes on the right at $\pm1$ and suitably enlarge the ones on the left, thus changing $\nu_{l}$ and $\nu_{r}$ without altering the embedded distribution $\nu=\nu_{l}+\nu_{r}$.
\end{remark}

Next, we turn to optimal Skorokhod embeddings for a reward function $G_{t}=g(t)$ that is deterministic and depends only on the time variable. More specifically, we assume the following.

\begin{assumption}\label{as:convexConcave}
  The reward function $G_{t}=g(t)$ is given by a bounded, Lipschitz continuous function $g: \R_{+}\to \R$ of time such that for some $t_{p}\geq0$, $g$ is differentiable on $(t_{p},\infty)$ and\footnote{The decrease/increase can be replaced by: $\partial^{+}g(s)<g'(t)$ for all $0\leq s<t_{p}<t$.}
\begin{align*}
    \mbox{$g$ is strictly convex and strictly decreasing}\quad \mbox{on}\quad &(0,t_{p}),\\
    \mbox{$g$ is strictly concave and strictly increasing}\quad \mbox{on}\quad & (t_{p},\infty).
  \end{align*}  
\end{assumption}

We also define $g(\infty)$ as the obvious limit. Such reward functions give rise to cave embeddings as follows.

\begin{proposition}\label{pr:caveOptimal}
  Suppose that $\nu(\{0\})=0$ and that $g$ satisfies Assumption~\ref{as:convexConcave}. There exists an optimal stopping time $\tau\in\cT(\nu)$; that is, we have $E[g(\tau)]=\sup_{\xi\in\cR(\nu)} \xi(g(\bt))$. Moreover, $\tau$ is the unique optimizer within~$\cR(\nu)$ and given by the first hitting time of a regular cave barrier with parting~$t_{p}$. The cave is unique among all regular caves. Finally, we have $P\{\tau=t\}=0$ for all $t\in[0,t_{p}]$.
\end{proposition}

\begin{proof}
  The first two assertions are stated in \cite[Theorem~2.5]{BeiglbockCoxHuesmann.14} which is itself a direct consequence of the monotonicity principle in~\cite[Theorem~1.3]{BeiglbockCoxHuesmann.14}. The latter states that if $\tau\in\cT(\nu)$ is optimal, then $(B_{\cdot\wedge\tau},\tau)\in\Gamma$ $\W$-a.s.\ for a Borel set $\Gamma\subseteq S$ satisfying 
\begin{equation}\label{eq:BCHmonotPrinciple}
  \mathsf{SG} \cap (\Gamma^{<}\times\Gamma)=\emptyset,
\end{equation}
where $\Gamma^{<}\subseteq S$ is defined by
$$
  \Gamma^{<}=\{(f',t')\in S:\, t'< t \mbox{ and } f'=f\mbox{ on }[0,t']\mbox{ for some }(f,t)\in \Gamma\}
$$
and $\mathsf{SG}$ is  the set of so-called stop-go pairs; cf.\ \cite[Definition 1.4]{BeiglbockCoxHuesmann.14}. In the present setting, due to the convexity and monotonicity properties of $g$, 
  \begin{equation}\label{eq:stopGoCave}
    \mathsf{SG}= \{((f,t),(f',t'))\in S\times S:\, f(t)= f'(t'),\, t'<t\leq t_{p} \mbox{ or }t_{p}\leq t <t'\}.
  \end{equation}
  (In \cite{BeiglbockCoxHuesmann.14} there are further assumptions on $g$ which are, however, not used in the proofs.) \cite[Theorem~2.5]{BeiglbockCoxHuesmann.14} also states that every optimal $\xi\in\cR(\nu)$ is the hitting time of a cave, and that implies uniqueness of $\tau$ similarly as in~\cite{Loynes.70}. Indeed, if $\tau_{1}$ and $\tau_{2}$ are optimizers we can define a randomized stopping time $\tau'$ by independently picking $\tau_{1}$ or $\tau_{2}$ with probability~$1/2$. But then $\tau'$ is optimal and it follows that $\tau'$ is non-randomized and thus $\tau_{1}=\tau_{2}$ a.s. The uniqueness of the cave holds because regular caves are in one-to-one correspondence with their hitting times, as noted above.
  
   For $t<t_{p}$ we have $P\{\tau=t\}=0$ since the left barrier cannot give rise to an atom in $\tau$. To see that $P\{\tau=t_{p}\}=0$, consider a measurable set $\Gamma\subseteq S$ with $P\{(B_{\cdot\wedge\tau},\tau)\in\Gamma\}=1$ and suppose that $P\{\tau=t_{p}\}>0$. Then $t_{p}>0$ and $\Gamma$ must contain a path $f$ which is stopped at $t_{p}$ and satisfies $x:=f(t_{p})\in (x_{min},x_{max})$. Thus, $l(x)<t_{p}$, and we have $(t',x)\in D$ for all $l(x)<t'<t_{p}$. It follows that there exists a stopped path $(f',t')\in \Gamma^{<}$ with $f'(t')=x$ and $t'<t_{p}$, but then $(f,t_{p})$ and $(f',t')$ form a stop-go pair by~\eqref{eq:stopGoCave} and now~\eqref{eq:BCHmonotPrinciple} yields a contradiction.
\end{proof}

The previous proposition leaves open how to \emph{characterize} the functions $l,r$ corresponding to the optimal barrier. Intuitively, the non-uniqueness of caves embedding $\nu$ stems from the fact that a given piece of $\nu$ can be absorbed at either of the two barriers. However, transferring mass from one to the other changes the reward $E[g(\tau)]$, and that is the basis of a variational characterization. Consider an optimal cave for $g$ at a point $x\in\R$ where both barriers absorb mass, or more precisely, $x\in\supp \nu_{l}\cap\supp\nu_{r}$. Intuitively, deforming $l$ and $r$ locally around $x$ corresponds to transferring mass from one boundary to the other, and if the cave is optimal, the derivative corresponding to this variation should vanish. Of course we cannot absorb a negative mass, so that if $x\in\supp \nu_{l}\setminus \supp \nu_{r}$ or vice versa, only one-sided variations are possible. Thus, the precise statement in Theorem~\ref{th:GammaCondition} below will consist of inequality conditions for each of the supports, amounting to an equality only on the intersection. 

It seems difficult to find a tractable parametrization for all variations of a cave that preserve the embedded distributions. Instead, we shall utilize the dual problem; a formal derivation runs as follows. Consider for simplicity the case $x\in\supp \nu_{l}\cap\supp\nu_{r}$ and recall that the dual problem admits a solution $(M,\psi)$. Since $g$ is Markovian (a function of time and state), one may expect that the martingale $M$ can also be chosen of the Markovian form $M_{t}=m(t,B_{t})$. Moreover, the dual solution satisfies $M_{\tau}+\psi(B_{\tau})=g(\tau)$ where $\tau$ is the exit time of~$D$, which roughly translates to $m(t,x)=g(t)-\psi(x)$ for $t\in\{l(x),r(x)\}$. Assuming a \emph{smooth fit} at the boundary, formally taking derivatives yields $\partial_{t}m(t,x)=g'(t)$ for $t\in\{l(x),r(x)\}$. Since~$M$ is a martingale, $m$ is a solution of the heat equation and then so is $\partial_{t}m$. A version of the  Feynman--Kac formula now yields that $\partial_{t}m(t,x)=E_{t,x}[g'(\tau)]$ for $l(x)\leq t\leq r(x)$. Thus, the difference
\begin{align*}
  g(r(x))-g(l(x)) &= m(r(x),x)+\psi(x)-m(l(x),x)-\psi(x)\\
  &= m(r(x),x)-m(l(x),x)
\end{align*}
can be expressed as $\int_{l(x)}^{r(x)} \partial_{t} m(t,x)\,dt= \int_{l(x)}^{r(x)} E_{t,x}[g'(\tau)]\,dt$.
The identity 
$$
  g(r(x))-g(l(x))= \int_{l(x)}^{r(x)} E_{t,x}[g'(\tau)]\,dt
$$
no longer refers to the dual solution. As mentioned above, this equation needs to be weakened to an inequality if~$x$ is not in the support of both measures~$\nu_{l},\nu_{r}$. 

In what follows we assume throughout that $\nu$ is a centered distribution with finite second moment and $\nu(\{0\})=0$. Some technical aspects of the proof depend on whether $\nu$ has atoms at the endpoints of its support. We focus on the case where $\nu$ has atoms at two endpoints: $\nu$ is concentrated on a compact interval $J=[x_{min},x_{max}]$ with $\nu(x_{min})>0$ and $\nu(x_{max})>0$. %
It is worth noting that given these properties of $\nu$, the regular caves under discussion satisfy $l(x)=t_{p}=r(x)$ at the two endpoints and the necessary jumps of $l$ and $r$ imply horizontal portions of $\partial D$ along $\{x=x_{min}\}$ and $\{x=x_{max}\}$; flat portions of floor and ceiling containing~$t_{p}$, so to speak. While $D$ is of bounded height, it can be unbounded to the right since $r(x)=\infty$ is a possible value.

In the following result we use the right derivative $\partial^{+}g$, but the same holds for the left derivative.

\begin{theorem}\label{th:GammaCondition}
  Let $D$ be a regular cave with parting $t_{p}$ embedding~$\nu$ and let $l,r$ be the functions delimiting $D$. Then $\tau=\inf\{t\geq 0:\, (t,B_{t})\notin D\}$ is the unique optimizer for $g$ in $\cR(\nu)$ if and only if
  \begin{align*}
    \int_{l(x)}^{r(x)} E_{t,x}[\partial^{+}g(\tau)]\,dt &\geq g(r(x)) - g(l(x)) \quad\mbox{for}\quad x\in\supp \nu_{l},\\
    \int_{l(x)}^{r(x)} E_{t,x}[\partial^{+}g(\tau)]\,dt &\leq g(r(x)) - g(l(x)) \quad\mbox{for}\quad x\in\supp \nu_{r},
  \end{align*}
  where $E_{t,x}[\partial^{+}g(\tau)]:=E[\partial^{+}g(\tau_{t,x})]$ for $\tau_{t,x}:=\inf\{s\geq t: \, (s,x+B_{s-t})\notin D\}$.
\end{theorem}

We use our setup for the dual problem from Section~\ref{se:dualProblem} with
$$
  T_{n}=T:=\inf \{t\geq0: B_{t}\in \partial J\}, \quad n\geq1.
$$
Since in this setting we are not interested in the dual martingale $M$ beyond time $T$, we redefine $\cD(G)$ slightly in the sense that $M$ is only defined up to $T$. Clearly, this does not affect the previous results, as one can trivially extend $M$ beyond $T$ in a constant fashion to retrieve a dual element in the previous sense.

Next, we consider the optimal cave $D$ and focus on the necessity of the variational condition. The first step in our proof is to construct the function $m(t,x)$. Importantly, we represent $m$ through an optimal stopping problem that will be used to derive the crucial relationship for $\partial_{t} m$.
We write $I$ for the interior of $J$; that is, $I=(x_{min},x_{max})$.

\begin{proposition}\label{pr:markovianM}
  There exists a dual optimizer $(M,\psi)\in\cD(G)$ with $M_{t}=m(t,B_{t})$ on $[0,\tau]$ for a universally measurable function $m: \R_{+}\times J\to\R\cup\{-\infty\}$ which is $C^{\infty}$ on $D$. Moreover, $m$ can be taken to be the value function of the optimal stopping problem with reward $g(t)-\psi(x)$.
\end{proposition}

\begin{proof} 
  Let $(\tilde{M},\psi)\in\cD(G)$ be any dual optimizer; cf.\ Corollary~\ref{co:dualExistence}. By Remark~\ref{rk:shift} we may assume that $\psi\geq0$. Let $\cT_{T}=\{\sigma\in\cT:\,\sigma\leq T\}$ and consider the optimal stopping problem 
  \begin{equation}\label{eq:MarkovOptStop1}
    \sup_{\sigma\in\cT_{T}} E[G^{\psi}_{\sigma}], \quad G^{\psi}_{t}:=g^{\psi}(t,B_{t}), \quad g^{\psi}(t,x):=g(t)-\psi(x).
  \end{equation}
  Below, it will be useful to see this as a Markovian optimal stopping problem in the sense of Mertens (we use the completed Brownian filtration throughout this proof). Indeed, setting $Y_{t}=(t,B_{t})$ we can define a process $\bar{Y}$ as the process $Y$ with absorption on the complement of $\R_{+}\times I$ in $\R_{+}\times \R$. Then~\eqref{eq:MarkovOptStop1} is equivalent to an infinite-horizon optimal stopping problem for the right-continuous Markov process $\bar{Y}$. 

  Next, we show that we can truncate $G^{\psi}$ from below without changing the value of~\eqref{eq:MarkovOptStop1}. Since $g$ is bounded we may assume that $g\geq0$.
  We then have $g(t)-\psi(x)\geq -\psi(x)$ and the value of $\sup_{\sigma\in\cT_{T}} E[-\psi(B_{\sigma})]$ with initial condition $B_{t}=x\in I$ is given by $-\psi^{**}(x)$, where $\psi^{**}\geq0$ is the convex hull of $\psi$ on the interval $J$. Thus, the value of the optimal stopping problem
  \begin{equation}\label{eq:MarkovOptStop2}
    \sup_{\sigma\in\cT_{T}} E[\bar{G}^{\psi}_{\sigma}], \quad \bar{G}^{\psi}_{t}:=\bar{g}^{\psi}(t,B_{t}),\quad \bar{g}^{\psi}(t,x) := (g(t)-\psi(x)) \vee (-\psi^{**}(x))
  \end{equation}  
  is the same as~\eqref{eq:MarkovOptStop1}. We have $\nu(\psi^{**})\leq \nu(\psi)<\infty$ by the definition of~$\cD(G)$ and due to the assumed form of $\nu$ this implies that $\psi^{**}$ is finite at the endpoints of the compact interval $J$; that is, $\psi^{**}$ is a bounded function and thus $\bar{g}^{\psi}(t,x)$ is bounded from below. As a result, the reward function $\bar{G}^{\psi}$ in~\eqref{eq:MarkovOptStop2} is of class~(D) and optional (but not necessarily right-continuous), putting us in the setting of \cite{Mertens.72, Mertens.73, ElKarouiLepeltierMillet.92}.
  
  Consider the Snell envelope $S$ of~\eqref{eq:MarkovOptStop2} as in \cite[Theorem~T4]{Mertens.72} or \cite{ElKarouiLepeltierMillet.92}; that is, $S$ is the minimal strong supermartingale satisfying $S\geq \bar{G}^{\psi}$ on $[0,T]$, and $S$ has the property that $S_{0}=E[S_{0}]=\sup_{\tau\in\cT_{T}} E[\bar{G}^{\psi}_{\tau}]$.
  (We use the symbol~$S$ as in the cited references; there should be no possibility of confusion with the space of stopped paths.)
  Noting that $\bar{G}^{\psi}$ is bounded, it follows that $S_{\cdot\wedge T}$ is bounded.
  As $(\tilde{M},\psi)\in\cD(G)$, the local martingale $\tilde{M}$ is another supermartingale satisfying $\tilde{M}\geq G^{\psi}\geq\bar{G}^{\psi}$, so that $S\leq \tilde{M}$ by minimality and in particular $S_{0}\leq \tilde{M}_{0}=0$. On the other hand, let $\tau$ be the primal optimizer. Then
  $$
   S_{0}\geq E[\bar{G}^{\psi}_{\tau}]\geq E[G_{\tau}-\psi(B_{\tau})]=E[G_{\tau}]-\nu(\psi)=0
  $$
  since $\tau\in\cT_{T}$ and there is no duality gap; cf.\ Theorem~\ref{th:duality}. As a result, $S_{0}=0$ and $\tau$ is an optimal stopping time for~\eqref{eq:MarkovOptStop2} and~\eqref{eq:MarkovOptStop1}. In particular, $S$ is a martingale on $[0,\tau]$.
   
  The strong supermartingale $S$ has a Mertens decomposition 
  $$
    S= M - A
  $$
  where $M$ is a local martingale, $M_{\cdot\wedge T}$ is of class~(D), $A$ is a predictable increasing process and $M_{0}=A_{0}=0$; cf.\ \cite[Appendix~I.20, p.\,414]{DellacherieMeyer.82}. We observe that $(M,\psi)\in\cD(G)$ is a dual optimizer. Indeed, the defining property of the Snell envelope shows that $M_{t}\geq S_{t}\geq G_{t}-\psi(B_{t})$ on $[0,T]$ and $M$ is uniformly bounded from below on $[0,T]$ since $M_{t}\geq S_{t}\geq-\psi^{**}(B_{t})$ as above. Finally, the optimality property depends only on $\psi$.
   
  The optimal stopping problem~\eqref{eq:MarkovOptStop2} is of a Markovian form as considered in \cite{ElKarouiLepeltierMillet.92, Mertens.73}. More precisely, \cite[Theorem~3.4]{ElKarouiLepeltierMillet.92} or \cite[Theorem~3]{Mertens.73} show that the Snell envelope is given by $S_{t}=v(t,B_{t})$ where $v$ is a universally measurable function, the smallest supermedian-valued function exceeding $\bar{g}^{\psi}$ (as defined below \cite[Theorem~1]{Mertens.73}), and that $v$ coincides with the value function of the optimal stopping problem in the Markovian sense. Since $S$ is a martingale on $[0,\tau]$, we have $M=S=v(\bt,B)$ on $[0,\tau]$ and thus $m:=v$ satisfies $m(t,B_{t})=M_{t}$ on $[0,\tau]$.

The smoothness of $m$ on $D$ follows from the martingale property of $v$ on~$D$ that is implied by the optimal stopping problem. Indeed, for $(t,x)\in D$, consider an open rectangle $R$ around $(t,x)$ whose closure is contained in $D$. Then $v|_{R}$ is smooth as it is the convolution with a smooth kernel, namely, $v(t,x)$ is the convolution of $v|_{\partial R}$ with the exit distribution of $(\bt,B)$ when $B$ is started at $(t,x)$.
\end{proof}

We can observe that the preceding arguments extend to more general functions $g(t,x)$ with a spatial dependence.

The following result connects $\partial_{t}m$ and $\partial^{+}g$ without going through a delicate smooth fit condition as in the sketch above. Instead, it exploits the representation of~$m$ through optimal stopping, and that is crucial in view of the barrier being non-smooth (for general~$\nu$).

\begin{lemma}\label{le:mAndgDerivatives}
  Let $x\in I$. Then 
  \begin{equation*}
    \partial_{t}m(t,x)=E_{t,x}[\partial^{+}g(\tau)]=E_{t,x}[\partial^{-}g(\tau)] \quad \mbox{for} \quad l(x)<t<r(x).
  \end{equation*}
\end{lemma}

\begin{proof}
  Let $x\in I$ and $l(x)< t<t_{1}<r(x)$. Following the proof of Proposition~\ref{pr:markovianM}, $m(t,x)$ is the value function of the stopping problem started at $(t,x)$ and the first exit time $\tau$ from $D$ is optimal. Define a stopping time $\tau_{1}$ for the initial condition $(t_{1},x)$ as the first exit time of the region $D_{1}:=D+\{(t_{1}-t,0)\}$; that is, the region $D$ translated to the right by $t_{1}-t$. Then $E_{t_{1},x}[\psi(B_{\tau_{1}})] = E_{t,x}[\psi(B_{\tau})]$ and hence
  \begin{align*}
     m(t_{1},x)-m(t,x) 
     &\geq E_{t_{1},x}[g(\tau_{1})-\psi(B_{\tau_{1}})] - E_{t,x}[g(\tau)-\psi(B_{\tau})] \\
     & = E_{t_{1},x}[g(\tau_{1})] - E_{t,x}[g(\tau)] \\
     & = E_{t,x}[g(\tau+t_{1}-t)] - E_{t,x}[g(\tau)] \\
     & = E_{t,x}[g(\tau+h) - g(\tau)]
  \end{align*}
  where $h=|t_{1}-t|$. Similarly, using $t_{1}<t$ yields $m(t,x)-m(t-h,x)\leq E_{t,x}[g(\tau) - g(\tau-h)]$. Recall that $m$ is differentiable. Dividing by $h$ and letting $h\downarrow 0$, dominated convergence then yields 
  $$
    \partial_{t}m(t,x)\leq  E_{t,x}[\partial^{-}g(\tau)] \leq E_{t,x}[\partial^{+}g(\tau)]\leq \partial_{t}m(t,x)
  $$
  where the middle inequality holds due to $\partial^{-}g\leq\partial^{+}g$; cf.\ Assumption~\ref{as:convexConcave}. The claim follows. We remark that $E_{t,x}[\partial^{-}g(\tau)] = E_{t,x}[\partial^{+}g(\tau)]$ can also be deduced directly from the properties of $\tau$; cf.\ Proposition~\ref{pr:caveOptimal}.
\end{proof}

\begin{proof}[Proof of Theorem~\ref{th:GammaCondition}, Necessity.]
  Let $D$ be optimal. For $x\in \{x_{min},x_{max}\}$ the result is trivial as $l(x)=r(x)$. 
  Given the properties of $l,r$ and $g$, it then suffices to show that 
    \begin{align*}
    \int_{l(x)}^{r(x)} E_{t,x}[\partial^{+}g(\tau)]\,dt &\geq g(r(x)) - g(l(x)) \quad\mbox{for}\quad \nu_{l}\mbox{-a.e.\ } x\in I,\\
    \int_{l(x)}^{r(x)} E_{t,x}[\partial^{+}g(\tau)]\,dt &\leq g(r(x)) - g(l(x)) \quad\mbox{for}\quad \nu_{r}\mbox{-a.e.\ } x\in I.
  \end{align*}
  Let $x\in I$ and $l(x)\leq l \leq r\leq r(x)$ with $r<\infty$. In view of Lemma~\ref{le:mAndgDerivatives} and the fundamental theorem of calculus,
  $$
    \int_{l}^{r} E_{t,x}[\partial^{+}g(\tau)]\,dt = \int_{l}^{r} \partial_{t}m(t,x)\, dt = m(r,x)-m(l,x).
  $$
  By the construction of $m$ we have $m(t,x)\geq g(t,x)-\psi(x)$ on $D$, while for $t=l(x)$ this inequality holds with equality for $\nu_{l}$-a.e.\ $x$ and for $t=r(x)$ it holds with equality for $\nu_{r}$-a.e.\ $x$.
  Focusing on the first case, we obtain $\nu_{l}$-a.e.\ that
  $$
    \int_{l(x)}^{r_{n}} E_{t,x}[\partial^{+}g(\tau)]\,dt \geq g(r_{n})-g(l(x))
  $$
  for a sequence $r_{n}\uparrow r(x)$, and the claim follows after recalling that $g$ is continuous on $[0,\infty]$ and $E_{t,x}[\partial^{+}g(\tau)]\geq 0$ for $t\geq t_{p}$. The $\nu_{r}$-a.e.\ inequality follows similarly.
\end{proof}

\begin{proof}[Proof of Theorem~\ref{th:GammaCondition}, Sufficiency]
  We only provide a sketch of the argument; the details are similar to the proof of \cite[Theorem~4.1]{CoxKinsley.18}. Suppose that the stated inequalities hold for $l,r$.
  Our goal is to construct a pair $(M,\psi)\in\cD(G)$ such that $M_\tau + \psi(B_\tau) = g(\tau)$ and $E[M_{\tau}]\leq0$ for the exit time~$\tau$ defined by $D$; this will imply the optimality of $\tau$ by Corollary~\ref{co:optimizerEquality}. For notational convenience, let 
$$
  h(t,x) = E_{t,x}[\partial^{+}g(\tau)],\quad \Gamma(x) = g(l(x)) - g(r(x)) + \int_{l(x)}^{r(x)} h(s,x)\, ds.
$$
As a first step, we consider the function
$$
  H(t,x) = g(r(x)) - \int_t^{r(x)} h(s,x) \,ds + \Gamma(x)^+
$$
and show that
\begin{align}
  g(t) &\leq H(t,x)\quad\mbox{for}\quad t\geq0,\quad x\in J, \label{eq:proofGammaSuff1}\\
  g(l(x)) &= H(l(x),x)\quad\mbox{for}\quad x \in \supp \nu_l, \label{eq:proofGammaSuff2}\\
  g(r(x)) &= H(r(x),x)\quad\mbox{for}\quad x \in \supp \nu_r. \label{eq:proofGammaSuff3}
\end{align}
Indeed, \eqref{eq:proofGammaSuff1} follows directly from the definitions. For $t\geq t_{p}$ we have that $g'(t)$ is decreasing and hence
$$
  g(t) = g(r(x)) - \int_{t}^{r(x)} \!\! \partial^{+}g(s) \,ds \leq g(r(x)) - \int_{t}^{r(x)} \!\! h(s,x) \,ds + \Gamma(x)^+ = H(t,x)
$$
whereas for $0\leq t < t_{p}$ we have $\partial^{+}g(s)\leq \partial^{+}g(u)$ for $t\leq s\leq u$ and hence
$$
	g(t) = g(l(x)) + \int_{l(x)}^t \partial^{+}g(s) \,ds \leq g(l(x)) + \int_{l(x)}^t h(s,x) \,ds \leq H(t,x).
$$
Moreover, the assumption in Theorem~\ref{th:GammaCondition} states that $\Gamma(x) \geq 0$ for $x \in \supp \nu_{l}$ and $\Gamma(x) \leq 0$ for $x \in \supp \nu_{r}$. This yields~\eqref{eq:proofGammaSuff2} and~\eqref{eq:proofGammaSuff3}.

The rest of the proof consists in showing that $H(t,B_{t})=M_{t}+\psi(B_{t})$ on $[0,\tau]$ and $H(t,B_t) \leq M_t + \psi(B_t)$ on $[0,T]$ for a local martingale $M$ with $M_{0}=0$ which is bounded below and a function $\psi\in L^{1}(\nu)$. Once that is achieved, \eqref{eq:proofGammaSuff1}--\eqref{eq:proofGammaSuff3} show that $(M,\psi)\in\cD(G)$ have the desired properties.

To obtain the decomposition we consider the function
$\bar m(t,x) := g(r(x)) - \int_t^{r(x)} h(s,x)\,ds$. One can show as in \cite{CoxKinsley.18} that
$\bar m(t,B_t)$ is a submartingale on~$[0,\tau]$ and conclude that there is a unique increasing predictable process $A$ on $[0,\tau]$ such that $\bar m(t,B_t) - A_t$ is a martingale on~$[0,\tau]$. One can further show that $A$ agrees with a continuous additive
functional on $[0,\tau]$ such that $\bar m(t,B_t) - A_t$ is well-defined on $[0,T]$ and a supermartingale.
Using a representation result for additive
functionals, one can therefore write $A_t = z(B_t) - z(B_0) - \int_0^t z_-'(B_s)\,dB_s$ for a convex function $z$ and conclude that
$\bar m(t,B_t) - z(B_t)$ is still a martingale on $[0,\tau]$ and a supermartingale on $[0,T]$. We can choose $z$ such that $z\geq z(0) = \bar m(0,0)$.

Set $\psi(x) = z(x)+\Gamma(x)^+$, so that 
$$
H(t,x) = \bar m(t,x) - z(x) + \psi(x).
$$
Moreover, let $M$ be the martingale part of the supermartingale $\bar m(t,B_t) - z(B_t)$.
To see that $(M,\psi)$ is indeed a dual element, we show that $\bar m$ and $z$ are bounded. Then it follows that $z(x) + \Gamma(x)^+$ is bounded (as $\Gamma$ is bounded under our assumptions) and that $\bar m(t,x) - z(x)$, and hence $M$, is bounded from below.
Indeed, boundedness of $\bar m$ follows from the identity
\[
  \bar m(t,x) = g(r(x)) - \int_t^{t_p} h(s,x)\1_{s < t_p} \,ds - \int_{t_p}^{r(x)}h(s,x)\1_{s > t} \,ds.
\]
The first two terms are trivially bounded. For the last term, observe that on the domain of integration we have $0 \leq h(s,x) \leq \partial^{+}g(s)$ which implies boundedness.
Next, suppose that $z$ is unbounded, then we must have 
$z(x_{min}) = +\infty$ or $z(x_{max}) = +\infty$ as $z$ is convex and bounded from below. Note that $B_\tau = x_{min}$ and $B_\tau = x_{max}$ with positive probability, but $E[\bar m(\tau,B_\tau) - z(B_\tau)] = m(0,0) - z(0) = 0$
by the martingale property. Therefore, $z$ must be bounded and the proof is complete.
\end{proof}

\section{Counterexamples}\label{se:counterex}

In this section we demonstrate that relaxing the regularity in the dual domain is necessary to achieve a complete duality theory for general reward functions. It is also shown that  the monotonicity principle fails without an integrability condition.

\subsection{Local Martingale Property of $M$}\label{se:localMartEx}

We construct an example showing that it is crucial to use local martingales $M$ rather than true martingales as in previous works. More precisely, we construct a continuous reward function $G$ such that for a wide class of marginals~$\nu$, any dual optimizer $(M,\psi)\in\cD(G)$ fails to have the true martingale property; in fact, $E[M_{t}]>0$ for all $t>0$.

Let $\nu$ be a centered distribution which is equivalent to the Lebesgue measure on $\R$ and satisfies $\nu(f)<\infty$, where
\begin{equation}\label{eq:strictLocalMartf}
  f(x)=\exp(x^{4}).
\end{equation}
We note that $J=\R$ and hence $\cup_{n}[0,T_{n}]=C_{0}(\R_{+})\times\R_{+}$ in Definition~\ref{de:dualDomain}.  An important ingredient for our reward function is the process
\begin{equation}\label{eq:strictLocalMartL}
  L_{t} := \exp\left(B^{4}_{t} - 2\int_{0}^{t} (3B_{s}^{2} + 4B^{6}_{s})\,ds\right), \quad t\geq0
\end{equation}
which can also be written as the stochastic exponential $L_{t} = \cE_{t}\left(\int_{0}^{\cdot} 4B^{3}_{s}\,dB_{s}\right)$. We are grateful to Johannes Ruf for indicating this remarkably simple example of a strict local martingale to us.

\begin{lemma}[J.~Ruf]\label{le:strictLocalMart}
  The stochastic exponential $L_{t} = \cE_{t}\left(\int_{0}^{\cdot} 4B^{3}_{s}\,dB_{s}\right)$ is a positive local martingale with $E[L_{t}]<1$ for all $t>0$. In particular, $L$ is not a martingale on $[0,t]$ for any $t\in(0,\infty)$.
\end{lemma}

We defer the proof to the end of this subsection. As our payoff function, we then choose 
$$
  G = 1-L + f(B);
$$
Note that $G$ is a continuous function on $S$. Moreover, it follows directly from~\eqref{eq:strictLocalMartf} and~\eqref{eq:strictLocalMartL} that $f(B_{t})\geq L_{t}$ and hence $G\geq1$. As $L\geq0$ and $f\in L^{1}(\nu)$, a particularly simple dual element is $(M,\psi):=(0,1+f)\in\cD(G)$; therefore, $\cD(G)\neq\emptyset$ and $\bI(G)\leq 1+\nu(f) <\infty$, so that the conditions of our main results in Sections~\ref{se:dualProblem} and~\ref{se:duality} are all satisfied. This dual element features a true martingale; however, it is not optimal. 

\begin{proposition}\label{pr:exampleLocMartOpt}
  (i) $(M,\psi):=(1-L,f)\in\cD(G)$ is optimal for the dual problem $\bI(G)$. Moreover, any $\xi\in\cR(\nu)$ is optimal for $\bS(G)$.
  
  (ii) If $(M,\psi)\in\cD(G)$ is any optimizer for $\bI(G)$, then $E[M_{t}]>0$ for all $t>0$, so that $M$ cannot be a martingale.
\end{proposition}

\begin{proof}
  (i) Let $\tau\in\cT(\nu)$ and $(M,\psi):=(1-L,f)\in\cD(G)$. It is clear that $M + \psi(B) = G$ $\xi$-a.s., and $\xi(M)\geq 0$ as $L$ is a nonnegative supermartingale with $L_{0}=0$. The claim now follows from Corollary~\ref{co:optimizerEquality}.
  
  (ii) Let $(M,\psi)\in\cD(G)$. We first prove that there exists $c\in\R$ such that
  \begin{equation}\label{eq:exampleLocMartOptBound}
    \psi(x)+cx\geq f(x), \quad x\in\R.
  \end{equation}
  Indeed, let $a,b\geq0$ and $\sigma=\inf\{t\geq0:\, B_{t}\notin (-a,b)\}$. Note that the local martingale $L_{\cdot\wedge\sigma}$ is bounded and hence a uniformly integrable martingale. On the other hand, as $\sigma\leq T_{n}$ for $n$ large enough, $M_{\cdot\wedge\sigma}$ must be bounded below (Definition~\ref{de:dualDomain}) and hence a supermartingale. In particular, $E[L_{\sigma}]=1$ and $E[M_{\sigma}]\leq0$, so that $M_{\sigma}+\psi(B_{\sigma})\geq G_{\sigma}=1-L_{\sigma}+f(B_{\sigma})$ implies
  $$
    E[\psi(B_{\sigma})]\geq E[f(B_{\sigma})].
  $$
  As $a,b$ were arbitrary, this yields as in the proof of Lemma~\ref{le:envelope} that the convex hull satisfies $(\psi-f)^{**}(0)\geq0$. Taking $c=\partial^{-}(\psi-f)^{**}(0)$, we have $(\bar\psi-f)^{**}\geq0$ for $\bar\psi(x)=\psi(x)+cx$ and the claim follows.
  
  In view of~\eqref{eq:exampleLocMartOptBound} and Remark~\ref{rk:shift} we may assume that $\psi\geq f$. If  $(M,\psi)$ is optimal, then $\nu(\psi)=\bI(G)=\nu(f)$ by~(i) and it follows that $\psi=f$ $\nu$-a.s.\ and hence Lebesgue-a.e. But now $M+\psi(B)\geq 1-L+ f(B)$ implies $M_{t}\geq 1-L_{t}$ $\W$-a.s.\ and in particular $E[M_{t}]\geq 1- E[L_{t}]$ for all $t>0$, so $E[M_{t}]>0$ by Lemma~\ref{le:strictLocalMart}.
\end{proof}

\begin{proof}[Proof of Lemma~\ref{le:strictLocalMart}.]
  (i) We first provide an auxiliary result. Let $W$ be a Brownian motion on a filtered probability space with measure $Q$. Then the SDE
  \[
    dX_{t} = 4X_{t}^{3}\,dt + dW_{t}, \quad X_{0}=0
  \]
  has a unique strong solution $X$ up to its explosion time $\tau$ and $Q\{\tau<T\}>0$ for all $T\in(0,\infty)$.
  Indeed, existence and uniqueness of~$X$ follow from the local Lipschitz continuity of the coefficients (see, e.g., \cite[Exercise 2.10, p.\,383]{RevuzYor.99}). The scale function of $X$ is $p(x)=\int_{0}^{x} e^{-2y^{4}}\,dy$ and the speed measure is $m(dx)=2e^{2x^{4}}\,dx$.   
Thus
\begin{align*}
  v(x) := \int_{0}^{x} (p(x)-p(y))\,m(dy) = \int_{0}^{x}\int_{y}^{x} e^{2(y^{4}-z^{4})}\,dz\,dy
\end{align*}
 and in particular $v(\infty):=\lim_{x\to\infty} v(x)$ is given by
\begin{align*}
  \int_{0}^{\infty}\int_{0}^{\infty} e^{2(y^{4}-(y+u)^{4})}\,du\,dy
  = \int_{0}^{\infty}\int_{0}^{\infty} e^{-2(u^{4}+4u^{3}y+6u^{2}y^{2}-4uy^{3})}\,du\,dy.
\end{align*}
Comparison with a Gaussian integral shows that this quantity is finite. 
In view of the symmetry $v(x)=v(-x)$, the same holds for $v(-\infty)$. Thus, Feller's test implies that both boundaries $\pm\infty$ are limit points of~$X$ in finite time with positive probability; in fact, the explosion time $\tau$ even satisfies $\tau<\infty$ $Q$-a.s.\ by \cite[Proposition~5.5.32, p.\,350]{KaratzasShreve.91}. Using the homogeneity of~$X$, this already implies that $Q\{\tau<T\}>0$ for all $T\in(0,\infty)$; see, e.g., \cite[Theorem~1.1]{BruggemanRuf.16} for an elegant argument.

  (ii) We can now prove the lemma. As an exponential of a continuous local martingale, it is clear that $L$ is a local martingale and strictly positive, hence a supermartingale. Let $T\in(0,\infty)$ and suppose for contradiction that $E[L_{T}]=1$, or equivalently, that $L$ is a martingale on $[0,T]$. Then we can introduce the equivalent probability $Q$ on $\cF_{T}$ via $dQ/d\W=L_{T}$ and Girsanov's theorem shows that the process $W_{t}:=B_{t}-4\int_{0}^{t}B^{3}_{s}\,ds$ is a $Q$-Brownian motion on $[0,T]$. Moreover, $B$ satisfies 
$dB_{t} = 4B_{t}^{3}\,dt + dW_{t}$ and $B_{0}=0$ under $Q$.
As shown in (i), this implies that $Q\{\tau<T\}>0$ for the explosion time $\tau$ of $B$, contradicting that the Brownian motion $B$ under $\W$ is non-explosive.
\end{proof}

\subsection{Regularity of $\psi$}

The following example shows that a duality gap can arise if the functions~$\psi$ in the dual domain $\cD(G)$ are required to be continuous. The reward $G=\1_{\Q}(\bt)$ was previously used in~\cite{GuoTanTouzi.15a} to illustrate that their duality result can fail when the reward function is irregular in time. In our framework, duality holds by Theorem~\ref{th:duality}. Nevertheless, it is instructive to detail the optimizers as this highlights the mechanics of our definitions.

\begin{example}\label{ex:rationalTime}
  Let $G=\1_{\Q}(\bt)$ and $\nu=(\delta_{-1}+\delta_{1})/2$.
  
  (i) We have $\bS(G)=\bI(G)=0$, a primal optimizer is given by the stopping time $\tau=\inf\{t\geq0:\, |B_{t}|=1\}$ and a dual optimizer is given by $M\equiv0$ and $\psi=\1_{(-1,1)}$.
  
  To see this, let $\tau,M,\psi$ be as above and note that $J=[-1,1]$. We choose $K_{n}=J$ and thus $T:=T_{n}=\tau$ for all $n$. We claim that $\psi(B)\geq G$ on $[0,T]$ up to evanescence. Indeed, $\psi(B)=1$ on $[0,T)$, thus $\{\psi(B)< G\}$ is contained in the graph of $T$ and of course also in $\{G=1\}$. But since $T$ has a continuous distribution, $\W\{T\in\Q\}=0$ and hence $[T]\cap\{G=1\}=\{T\in\Q\}\times\R_{+}$ is indeed negligible up to evanescence. As a result, $(M,\psi)\in\cD(G)$. In view of $\nu(\psi)=0$ and $E[G_{\tau}]=0$, the optimality of $(M,\psi)$ and $\tau$ now follows from Corollary~\ref{co:optimizerEquality}.
  
  (ii) If the dual domain is restricted to continuous functions $\psi$, a duality gap arises:  the dual problem over continuous $\psi$ has value~$1$ instead of~$0$. Indeed, let $(M,\psi)\in\cD(G)$ be such that $\psi$ is continuous. We claim that there exists $c\in\R$ such that $\psi(x)\geq 1+cx$ for all $x\in[-1,1]$; in particular, this will imply that $\nu(\psi)\geq1$. Let $\sigma<T$ be a stopping time, let $\sigma_{n}'=\inf\{t\geq \sigma:\, t\in 2^{-n}\N\}$ be the usual dyadic approximation $\sigma_{n}'\downarrow \sigma$ and let $\sigma_{n}=\sigma_{n}'\wedge T$. Then we have 
\[
  M_{\sigma_{n}}+\psi(B_{\sigma_{n}})\geq G_{\sigma_{n}}=1 \quad\mbox{on}\quad \{\sigma_{n}<T\}
\]
since $\sigma_{n}$ has rational values on $\{\sigma_{n}<T\}$. As $\W\{\sigma_{n}<T\}\to1$, the continuity of $\psi$ and $M$ yields that  $M_{\sigma}+\psi(B_{\sigma})\geq1$ and hence $E[\psi(B_{\sigma})]\geq1$. Since this holds in particular for the hitting time $\sigma$ of any set $\{-a,b\}$ where $-1<-a\leq0\leq b<1$, it follows that $\psi^{**}(0)\geq0$ where $\psi^{**}$ is the convex hull on $(-1,1)$. Let $c=\partial^{-}\psi^{**}(0)$, then using again the continuity it follows that $\psi(x)\geq 1+cx$ for all $x\in[-1,1]$, as claimed.
\end{example}

\subsection{Monotonicity Principle}\label{se:monPrinciple}

In this section we show that the monotonicity principle of Corollary~\ref{co:monotPrinciple} does not hold
without an integrability condition. Indeed, we have the following.

\begin{proposition}\label{pr:noMonotPrinciple}
  There exist a Borel function $G: S\to [0,\infty)$, a centered distribution $\nu$ on~$\R$ with all moments finite and $\bS(G)<\infty$, and randomized stopping times $\xi^{1},\xi^{2}\in\cR(\nu)$ which are equivalent as measures on~$S$, such that $\xi^{1}$ is optimal for $\bS(G)$ and $\xi^{2}$ is not optimal for $\bS(G)$. 
\end{proposition} 

As $\xi^{1}(\Gamma)=1$ is equivalent to $\xi^{2}(\Gamma)=1$, for any Borel set $\Gamma\subseteq S$, it follows that \emph{optimality of $\xi\in\cR(\nu)$ cannot be determined by its support}.

We start with some preliminary results that will be used in the construction. Recall that $\xi\in\cR$ is defined as a measure on $C_{0}(\R_{+})\times \R_{+}$ and induces a measure on~$\cB(S)$ via $\xi(\Gamma):=\xi\{(\omega,t)\in C_{0}(\R_{+})\times \R_{+}:\,  (\omega|_{[0,t]},t)\in\Gamma\}$; in fact, $\xi$ is completely characterized by the latter. A product measure $\W \otimes \lambda$ on $C_{0}(\R_{+})\times \R_{+}$ induces a measure on $S$ in the same fashion.
In what follows, we set $f(x) = \exp(x^4)$ and denote by $L$ the strict local martingale
defined in~\eqref{eq:strictLocalMartL}.

\begin{lemma}\label{lem:smallStoppingTime}
There exists $\hat\xi^{\rm int}\in\cR$ such that $\hat\xi^{\rm int}(f(B))+\hat\xi^{\rm int}(\bt)<\infty$ and
$\hat\xi^{\rm int}\gg\W \otimes \lambda$ on~$S$, where $\lambda$ is the Lebesgue measure.
\end{lemma}

\begin{proof}
For $n\geq 1$ we define $\tau^n = \inf\{t:\, |B_t| \geq n\}$ and 
\[
  S^n = \{(\omega,t)\in S :\, \sup_{s \leq t} |\omega_s| < n \} = \{(\omega,t)\in S :\,t<\tau^n(\omega)\}.
\]
We first construct $\xi^n\in\cR$ such that
\begin{equation}\label{eq:smallStoppingTimeStep}
  \xi^{n} [0,\tau^{n}] =1 \qquad\mbox{and}\qquad \xi^{n} \gg \W \otimes \lambda \;\mbox{ on }\; S^{n} .
\end{equation} 
Consider the adapted, increasing process $A^{n}$ defined by 
\[
  A_t^n := 1 - e^{-t}\1_{t < \tau^n};
\]
it is strictly increasing and differentiable up to $\tau^{n}$ and then jumps to the value 1. Thus, the kernel
\[
  \xi^n_\omega(dt) = dA_t^n(\omega) = e^{-t}\1_{t < \tau^n(\omega)}\,dt + e^{-\tau^n(\omega)}\delta_{\tau^n(\omega)}(dt)
\]
defines a randomized stopping time via $\xi^{n}=\W(d\omega)\xi^n_\omega(dt)$ which satisfies~\eqref{eq:smallStoppingTimeStep}.

Clearly~\eqref{eq:smallStoppingTimeStep} implies $\xi^n(f(B)) \leq \exp(n^4)$ and  $\xi^n(\bt) \leq E[\tau^n] = n^2\leq \exp(n^4)$.
Let $(a_n)_{n\geq1}$ be a sequence in $(0,1)$ such that $\sum_{n \geq 1} a_n = 1$
and $\sum_{n \geq 1} a_n \exp(n^4) < \infty$. We define
$\xi := \sum_{n \geq 1} a_n \xi^n$; then $\xi\in\cR$ satisfies $\xi^n(f(B))<\infty$ and  $\xi^n(\bt)<\infty$. 
Since every stopped path is bounded, we have $\cup_{n \geq 1} S^n = S$ and thus~\eqref{eq:smallStoppingTimeStep} implies $\xi \gg \W \otimes \lambda$ on~$S$ as desired.
\end{proof}

\begin{lemma}\label{lem:payoffRandomization}
Let $\sigma = \inf \{ t :\, t^2 + B_t^2 = 1 \}$. There exists an $\cF_{\sigma}$-measurable
Bernoulli random variable~$X$ independent of $B_\sigma$.
\end{lemma}

\begin{proof}
Define $\sigma' = \inf \{ t :\, t^2 + B_t^2 = 1/2 \}$. Then $B_{\sigma'}$ is $\cF_{\sigma}$-measurable and its conditional distribution given $B_{\sigma}$ is atomless. In particular, there exists a conditional
median $m(x)$ given $B_\sigma=x$; that is, $P[B_{\sigma'} \geq m(x)|B_\sigma = x] = 1/2$.  By construction, $X:= \1_{B_{\sigma'} \geq m(B_\sigma)}$ has a Bernoulli distribution and is independent of~$B_\sigma$.
\end{proof}

\begin{lemma}\label{lem:trueMartingaleCondition}
We have $\xi(L) \leq 1$ for all $\xi\in\cR$. If $\xi$ embeds a distribution $\nu$ with $\nu(f) < \infty$, then $\xi(L) = 1$.
\end{lemma}

\begin{proof}
As $L$ is a positive local martingale with $L_0 = 1$, we have $\xi(L) \leq 1$ by Fatou's lemma. Suppose that $\xi\in\cR(\nu)$ where $\nu(f) < \infty$ and recall that $0 \leq L_t \leq f(B_t)$. As $f$ is convex, $f(B_{t})$ is a positive submartingale up to~$\xi$ and hence of class~(D), where we may use the representation of $\xi$ as nonrandomized stopping time in the enlarged filtration (cf.\ Lemma~\ref{le:rho}) to apply the standard results of stochastic analysis. This implies that $L$ is a martingale of class~(D) up to $\xi$ and in particular $\xi(L) = 1$.
\end{proof}

\begin{proof}[Proof of Proposition~\ref{pr:noMonotPrinciple}]
Let $G\geq0$ be the payoff defined by
$$
  G_t = \1_{t > \sigma,X=0} + L_{t-\sigma}(B^{\sigma\mapsto})\1_{t > \sigma,X = 1}
$$
where $\sigma= \inf\{t:\,t^2 + B_t^2 = 1\}$, the random variable $X$ is as in Lemma~\ref{lem:payoffRandomization} and, with the notation of Definition~\ref{de:gluedAndShift},
$$
  L_{t-\sigma}(B^{\sigma\mapsto}) = \exp\left((B_t-B_\sigma)^4 - 2 \int_\sigma^t (3(B_s - B_\sigma)^2 + 4 (B_s-B_\sigma)^6)\,ds\right).
$$

Let $\xi^{\rm int}=\sigma\oplus\hat\xi^{\rm int}$ be the randomized stopping time obtained by shifting $\hat\xi^{\rm int}$ of Lemma~\ref{lem:smallStoppingTime} by~$\sigma$ (cf.\ Definition~\ref{de:gluedAndShift}); that is, if $\hat\xi^{\rm int}=\W(d\omega)\hat\xi^{\rm int}_{\omega}(ds)$, then 
$\xi^{\rm int}=\W(d\omega)\xi^{\rm int}_{\omega}(ds)$ where 
$$
\xi^{\rm int}_{\omega}[0,t]=\1_{\sigma(\omega)<t} \;\hat\xi^{\rm int}_{\omega_{\cdot}-\omega_{\sigma(\omega)}}[0,t-\sigma(\omega)].
$$
Using Lemma~\ref{lem:trueMartingaleCondition}, we then have $\xi^{\rm int}(G)=\hat\xi^{\rm int}(L)=1$.

Next, let $\hat\xi^{\rm exp}$ be an exponential random time, defined by its kernel $\hat\xi^{\rm exp}_{\omega}(ds)=e^{-s}\,ds$. As $E[L_{t}]<1$ for $t>0$ by Lemma~\ref{le:strictLocalMart}, we have $\hat\xi^{\rm exp}(L)=\int_{0}^{\infty} e^{-t} E[L_{t}]\,dt <1$. Moreover, $\hat\xi^{\rm exp}$ is clearly equivalent to $\W\otimes\lambda$ on $S$, so that $\hat\xi^{\rm exp}\ll \hat\xi^{\rm int}$. The reverse is not true, but if we set
$$
  \hat\xi^{\rm avg} = (\hat\xi^{\rm exp} + \hat\xi^{\rm int})/2,
$$
then $\hat\xi^{\rm avg}$ is equivalent to $\hat\xi^{\rm int}$ on $S$ and we also have $\hat\xi^{\rm avg}(L)<1$.

Define $\xi^{\rm avg}=\sigma\oplus\hat\xi^{\rm avg}$.
Then $\xi^{\rm avg}$ and $\xi^{\rm int}$ already have properties close to the desired ones, but they do not embed the same distribution yet. To achieve that, we randomly mix the two stopping times using~$X$. Indeed, we define the randomized stopping times $\xi^{1}$ and $\xi^{2}$ through their kernels
\[
  \xi^{1}_{\omega} := \xi^{\rm avg}_{\omega} \1_{X(\omega)=0} + \xi^{\rm int}_{\omega}\1_{X(\omega)=1}, \qquad \xi^{2}_{\omega} := \xi^{\rm avg}_{\omega}\1_{X(\omega)=1} + \xi^{\rm int}_{\omega}\1_{X(\omega)=0}.
\]
Then
$$
  \xi^{1}(G)=\frac12 \hat\xi^{\rm avg}(1) + \frac12 \hat\xi^{\rm int}(L)=1
$$
whereas 
$$
  \xi^{2}(G)=\frac12 \hat\xi^{\rm avg}(L) + \frac12 \hat\xi^{\rm int}(1)<1.
$$
By construction, $\xi^{1}$ and $\xi^{2}$ are equivalent on $S$ and embed the same distribution~$\nu$. The integrability properties of $\hat\xi^{\rm exp}$ and $\hat\xi^{\rm int}$ entail that $\nu$ has all moments (even some exponential moments) and that $\xi^{1}(\bt)=\xi^{2}(\bt)<\infty$.
\end{proof}

\appendix

\section{Appendix: Extension to Finite First Moment}\label{se:firstMoment}

In the body of this article we have assumed that the embedded measure $\nu$ has a finite second moment, but the results can be extended to the case of a finite first moment by using well-known facts, at the expense of a slightly more convoluted definition. The crucial observation is that under the second moment condition, $E[\tau]<\infty$ for $\tau\in\cT(\nu)$ is equivalent to $\tau$ being minimal (or uniformly integrable), and minimality can still be used in the first moment case. For a randomized stopping time, the simplest definition is obtained by referring to its stopping time representation in a larger filtration.

For the remainder of this section, $\nu$ is a centered distribution on $\R$ with finite first moment.

\begin{definition}\label{de:minimal}
  Let $\tau$ be a finite stopping time such that $B_{\tau}\sim\nu$. Then $\tau$ is called \emph{minimal} if there exists no smaller embedding; that is, if $\sigma$ is another stopping time such that $B_{\sigma}\sim\nu$ and $\sigma\leq\tau$ a.s., then  $\tau=\sigma$ a.s. We denote by $\cT(\nu)$ the set of all such $\tau$.
  
  Let $\xi$ be a randomized stopping time such that $\xi\circ B^{-1}=\nu$ and let $\rho$ be the associated $\bar\F$-stopping time; cf.~Lemma~\ref{le:rho}. Then $\xi$ is called \emph{minimal} if~$\rho$ is minimal in the above sense. We denote by $\cR(\nu)$ the set of all such $\xi$.
\end{definition}

We can now state the announced extension.

\begin{remark}\label{rk:extensionFirstMoment}
 The results in Sections~\ref{se:approximation}--\ref{se:duality} continue to hold under the first moment condition.
\end{remark}

\begin{proof}
  The modifications are as follows. (i) In the proof of Lemma~\ref{le:EasyTesting} we used that $\cR(\nu)$ is weakly compact. This still holds in the present setting; cf.\ \cite[Section~7.1]{BeiglbockCoxHuesmann.14}.
(ii) A stopping time $\tau$ is minimal if and only if $B_{\cdot \wedge \tau}$ is uniformly integrable; cf. \cite[Theorem~3]{Monroe.72}. Thus, in the proofs of Lemmas~\ref{le:superhedUptoEmbedding} and~\ref{le:Msupermart}, we can use directly that
$B_{\cdot \wedge \tau}$ is uniformly integrable instead of deriving this from $E[\tau] < \infty$. (iii) Proposition~\ref{pr:strongDualityCont} still holds, e.g., by \cite[Theorem~7.2 and Example~7.2.1]{BeiglbockCoxHuesmann.14}. (iii) The proof of Lemma~\ref{le:SisCapacity} again used that $\cR(\nu)$ is weakly compact; cf.~(i). 
The other proofs did not use directly that $\nu$ has finite second moment.
\end{proof}

\newcommand{\dummy}[1]{}

\end{document}